\documentclass[reqno]{amsart}
\usepackage[foot]{amsaddr}
\usepackage{enumerate,enumitem} 
\usepackage[margin=1.25in]{geometry}
\usepackage{amsmath,amsthm,amscd,amssymb}
\usepackage{latexsym}
\usepackage{color}
\usepackage{todonotes}
\usepackage{multirow} 
\usepackage{mathrsfs}
\usepackage{float}
\usepackage{hyperref,cleveref}
\usepackage{indentfirst}
\usepackage[skip=5pt plus1pt, indent=15pt]{parskip}
\usepackage{bbm} 
\usepackage{mathtools} 
\usepackage{calrsfs} 
\usepackage{ulem} 
\usepackage[all]{xy}

\usepackage{tikz,pgfplots}
\usetikzlibrary {arrows.meta,bending,positioning}
\pgfplotsset{compat=1.18}
\usepgfplotslibrary{fillbetween}

\hypersetup{
    bookmarks=true,
    colorlinks=false,
    citebordercolor=0 1 0, 
    linkbordercolor=1 1 1, 
    pdftitle={Stochastic Process}
}

    \usepackage{etoolbox}
    \BeforeBeginEnvironment{ex}{\vspace{10pt}}
    \BeforeBeginEnvironment{definition}{\vspace{10pt}}
    \BeforeBeginEnvironment{theorem}{\vspace{10pt}}
    \BeforeBeginEnvironment{remark}{\vspace{10pt}}
    \BeforeBeginEnvironment{assumption}{\vspace{10pt}}
    \BeforeBeginEnvironment{proposition}{\vspace{10pt}}
    \BeforeBeginEnvironment{lemma}{\vspace{10pt}}
    \BeforeBeginEnvironment{corollary}{\vspace{10pt}}
    \BeforeBeginEnvironment{Question}{\vspace{10pt}}
    \BeforeBeginEnvironment{exercise}{\vspace{10pt}}

\numberwithin{equation}{section}

\newtheorem{theorem}{Theorem}[section]

\newtheorem{corollary}[theorem]{Corollary}
\newtheorem{lemma}[theorem]{Lemma}
\newtheorem{proposition}[theorem]{Proposition}
\newtheorem{assumption}[theorem]{Assumption}

\newtheorem{remark}[theorem]{Remark}

\DeclareMathAlphabet{\pazocal}{OMS}{zplm}{m}{n}

\newcommand{\ds}{\displaystyle}

\newcommand{\RomanNum}[1]{\uppercase\expandafter{\romannumeral #1\relax}}

\newcommand{\e}{\varepsilon}

\newcommand{\eps}{\epsilon}

\newcommand{\R}{\mathbb{R}}

\newcommand{\N}{\mathbb{N}}

\newcommand{\Ba}{\mathcal{B}}

\newcommand{\Cb}{\pazocal{C}}

\newcommand{\Eb}{\pazocal{E}}
\newcommand{\Fa}{\mathcal{F}}

\newcommand{\Wb}{\pazocal{W}}

\newcommand{\Zb}{\pazocal{Z}}

\newcommand{\Exp}{\mathbb{E}}
\newcommand{\Prob}{\mathbb{P}}

\newcommand{\CondExp}[2]{\Exp[#1\,|\,#2]}

\newcommand{\Indicator}[1]{\mathbbm{1}_{#1}}
\newcommand{\InnerProd}[2]{\langle #1,\,#2 \rangle}

\newcommand{\terminology}[1]{\textbf{\textit{#1}}}

\def\XXint#1#2#3{{\setbox0=\hbox{$#1{#2#3}{\int}$}
        \vcenter{\hbox{$#2#3$}}\kern-.5\wd0}}

\hypersetup{
    bookmarks=true,
    colorlinks=false,
    citebordercolor=0 1 0, 
    linkbordercolor=1 1 1, 
    pdftitle={Convergence Rate for the Last Iterate of Stochastic Gradient Schemes}
}

\begin{document}

\title{Convergence Rate for the Last Iterate of Stochastic Gradient Descent Schemes}


\begin{abstract}
We study the convergence rate for the last iterate of stochastic gradient descent (SGD) and stochastic heavy ball (SHB) in the parametric setting when the objective function $F$ is globally convex or non-convex whose gradient is $\gamma$-H\"{o}lder. Using only discrete Gronwall's inequality without Robbins-Siegmund theorem, we recover results for both SGD and SHB: $\min_{s\leq t} \|\nabla F(w_s)\|^2 = o(t^{p-1})$ for non-convex objectives and $F(w_{\tau \wedge t}) - F_* = o(t^{2\gamma/(1+\gamma) \cdot \max(p-1,-2p+1)-\eps})$ for $\beta \in (0, 1)$, $\tau := \inf \{ t > 0 : F(w_t) = F_*\}$, and $\min_{s \leq t} F(w_s) - F_* = o(t^{p-1})$ for convex objectives $F$ whose minimum is $F_*$. In addition, we proved that SHB with constant momentum parameter $\beta \in (0, 1)$ attains a convergence rate of $F(w_t) - F_* = O(t^{\max(p-1,-2p+1)} \log^2 \frac{t}{\delta})$ with probability at least $1-\delta$ when $F$ is convex and $\gamma = 1$ and step size $\alpha_t = \Theta(t^{-p})$ with $p \in (\frac{1}{2}, 1)$.
\end{abstract}


\keywords{stochastic optimization, stochastic gradient descent, stochastic heavy ball, almost sure convergence, convergence rate with high probability}

\author{Marcel Hudiani$^{\dagger}$}

\address{
    $^{\dagger}$ Department of Mathematics\\
    University of Arizona\\
    621 N. Santa Rita Ave.\\
    Tucson, AZ 85721, USA
    }

\email{{\tt marcelh@arizona.edu}}

\maketitle


%

\section{Introduction}
\label{s:1}

We study the almost sure convergence rate for the last iterate of Stochastic Gradient Descent (SGD) and Stochastic Heavy Ball (SHB) schemes associated to solving an unconstrained optimization problem involving a cost function defined on a subset of $\R^d$. We consider the application of these algorithms in the following context as in \cite{Bottou2012}:
\begin{equation}
    w_{t+1} - w_t = -\alpha_t \,\nabla \ell(Z_t, w_t) + \beta (w_t - w_{t-1}) ~,~ \beta \in [0, 1)
    \label{eq:shb}
\end{equation}
where $w_t$ is in $\R^{d}$, $Z_t \in \R^{N}$ is an i.i.d process with finite variance sampled from a continuous distribution with density $\rho$ independent of $\Fa_t = \sigma(w_s : s \leq t)$, $\alpha_t$ is the (non-random) step size, and $\ell$ is an estimator of a deterministic cost function $F : \Wb \rightarrow [0, \infty)$ defined by $F(w) = \Exp_{\rho}[\ell(Z, w)]$. When $\beta = 0$, (\ref{eq:shb}) is the SGD and can be rewritten as
\begin{equation}
    w_{t+1} - w_t = -\alpha_t \nabla F(w_t) + \alpha_t\,\delta m_t ~~~,~~~ \delta m_t = \nabla F(w_t) - \nabla\ell(Z_t,w_t).
    \label{eq:sgd2}
\end{equation}

SHB \eqref{eq:shb} has been studied quite extensively, especially when $F$ is coercive, convex, and smooth with bounded second derivative and with finite variance in the noise. Specifically, under such conditions, results on almost sure convergence, scaling limit in distribution, and $L^2$ convergence rates have been derived \cite[theorem 1-2,4,6]{gadat2016shb}. Similar results can also be found for the special case \eqref{eq:sgd2}. Such a scheme is a class of \terminology{Robbins-Monro} procedures \cite{robbins1951}. See also \cite{kiefer1952}. Under some conditions on the noise and $\nabla F$, Robbins-Monro type algorithms converge almost surely \cite[section \RomanNum{1}.1, thm \RomanNum{1}.9]{ljung1992}. Such a statement relies on the \terminology{Robbins-Siegmund theorem} \cite[theorem 1]{robbins1971}, which is a convergence theorem for non-negative \terminology{almost super-martingale}. We remark that other methods to treat the almost sure convergence of Robbins-Monro algorithm exists, e.g. the \terminology{Ordinary Differential Equation (ODE) method}, which relies on the stability theory of ODE. See \cite[section \RomanNum{1}, page 11]{ljung1992}.

Going from classical theory, there are two important directions to pursue. The first belongs to the case where $F$ is non-convex with Lipschitz gradient (weakening the convexity and smoothness assumptions on $F$) under weakened assumption on the noise. The second belongs to almost sure convergence rates. Despite many classical results, strong convergence behavior of stochastic gradient schemes under general convexity assumptions on $F$ (e.g. non-convex with saddle points) under weaker noise assumption has been studied fairly recently. Liu and Yuan provided sufficient conditions on $F$ without assuming global uniform variance bound on the estimator $\|\nabla \ell\|$ for SHB and SNAG (Stochastic Nesterov Accelerated Gradient) to avoid saddle points almost surely \cite{liuyuan2024}. In addition, the authors of \cite{liuyuan2024} expanded the application of Robbins-Siegmund theorem for proving almost sure convergence rates for SGD, SHB, and SNAG for non-convex, convex, and strongly convex objectives whose gradient is Lipschitz \cite[table 1]{liuyuan2024}. We remark that other recent convergence rate results include \cite{agarwal2012,lei2018,nguyen2019new,orabona2020,sebbouh21a,weissmann2025}.

When $F$ is convex, more results are available. The optimal convergence rate $O(t^{-1/2})$ for stochastic first order oracle was proved for bounded, convex, and Lipschitz cost functions in \cite[thm 1-2]{agarwal2012}. A constructive proof for such an optimal rate using SHB is done in \cite{sebbouh21a}. More precisely, Sebbouh et al. showed that there exists $\{(\alpha_t, \beta_t) : t \geq 0 \}$ such that \cref{eq:shb} attains the optimal convergence $O(t^{-\frac{1}{2} + \eps})$ for convex cost functions with Lipschitz gradient \cite[theorem 13]{sebbouh21a}. In this case, the momentum $\beta_t$ is a function of the step size $\alpha_t$. When the step-size is chosen to be $\alpha_t = O(t^{-1/2})$, the momentum takes values $\beta_t = O(0.5 t/(t+1))$. Intuitively, when the minimizer is still far away in the early iterations, the algorithm is approximately an SGD, but gradually transitions to SHB with large momentum as $t \rightarrow \infty$. This shows that the choice of $\beta_t$ is crucial since the almost sure asymptotic performance of SHB is the same as SGD when $\beta_t = \beta$ is held constant \cite{liuyuan2024}. In the non-parametric setting, Lei, Shi, and Guo provided almost sure convergence rate and fluctuation for SGD ($\beta_t = 0$) when $\nabla F$ is only $\gamma$-H\"{o}lder \cite{lei2018}. More precisely, SGD achieves an almost sure convergence rate $o(t^{\max(p-1,1-p(\gamma+1))+\eps})$ given step size $O(t^{-p})$ with $p \in (\frac{1}{1+\gamma}, 1)$ for a convex objective whose gradient is $\gamma$-H\"{o}lder \cite[theorem 6]{lei2018}. Moreover, SGD converges at rate $F(w_{t+1}) - F_* = O (t^{\max(p-1,1-p(\gamma+1))} \, \log^2(t/\delta))$ with probability at least $1-\delta$ if $p \in (\frac{1}{1+\gamma}, 1)$ \cite[corollary 11]{lei2018}.

\subsection{Contribution}

We will restrict the scope of our study to the case where $\beta_t := \beta$ is constant in \eqref{eq:shb} and $\ell(z,\cdot)$ is convex and $\nabla \ell(z,\cdot) \in \Cb^{\gamma}(\R^d)$ with H\"{o}lder constant $L > 0$ and $\gamma \in (0, 1]$. Further, we assume that there exists a global minimum $F_* := \min_{w \in \R^d} F(w) \geq 0$ and a global minimizer $w_* \in \R^d$. In the landscape of convergence rate theory for stochastic gradient schemes, the survey of prior works provided above indicates that the case of a convex objective whose gradient is $\gamma$-H\"{o}lder for SHB is unexplored. This is where our contribution lies.

There are three main contribution from our work. First, we provide an alternative method to prove convergence rate using Gronwall's inequality and Doob's martingale convergence theorem instead of using Robbins-Siegmund theorem (cf. \cref{p:weak_RS} and \cref{l:weak_RS_rate}). Second, we provide an almost sure convergence rate result for SHB for convex objective whose gradient is $\gamma$-H\"{o}lder. As stated previously, this case is unexplored. Roughly stated, the result says that given a convex objective with $\gamma$-H\"{o}lder gradient, SHB with momentum $\beta \in (0, 1)$ and learning rate (step size) $t^{-p}$ with $p \in (\frac{1}{1+\gamma}, 1)$ converges to a global minimum almost surely with rate $\frac{2\gamma}{1+\gamma} \max(p-1, 1-p(\gamma + 1))$ (cf. \cref{t:risk_o}). Finally, we also provide convergence rate in probability for the case $\gamma = 1$ (cf. \cref{t:risk_hp}). This is also unexplored as the case $\gamma = 1$ was only done for SGD. Our result is consistent with \cite{liuyuan2024,lei2018} for the case $\gamma = 1$. \Cref{tab:thesis:sgd:intro:rates} compares our convergence rate results with prior work.

\begin{table}[h!]
    \centering
    \begin{tabular}{|p{2.8cm}|p{4.25cm}|p{5.5cm}|}
        \hline
        \multirow{2}{*}{}
        Algorithm & Cost Function and Learning Rate & Convergence Rate\\
        &  &
        \\
        \hline
        \hline
        \multirow{2}{*}{}
        SGD-np & Convex, $\gamma$-H\"{o}lder gradient & $O(t^{\max(p-1,1-p(\gamma + 1))})$\\
        \cite{lei2018} & $O(t^{-p}), p \in (\frac{1}{1+\gamma}, 1)$ & a.s. and h.p..
        \\
        \hline
        \multirow{2}{*}{}
        SHB & Convex, Lipschitz gradient &\\
        \cite{sebbouh21a} & $O(t^{-\frac{1}{2}})$, $\beta_t \in O(\frac{t}{1+t})$ & $O(t^{-\frac{1}{2}+\eps})$ a.s..
        \\
        \hline
        \multirow{2}{*}{}
        SGD, SHB, SNAG & Convex, Lipschitz gradient &\\
        \cite{liuyuan2024} & $O(t^{-\frac{2}{3}})$ & $O(t^{-\frac{1}{3}})$ a.s..
        \\
        \hline
        \multirow{2}{*}{}
        SHB, $\beta \in (0, 1)$ & Convex, $\gamma$-H\"{o}lder gradient & $O(t^{\frac{2\gamma}{1+\gamma} \max(p-1,1-p(\gamma + 1))})$ a.s.
        \\
        & $O(t^{-p}), p \in (\frac{1}{1+\gamma}, 1)$ & $O(t^{\max(p-1,-2p+1)})$ h.p., $\gamma = 1$.
        \\
        \hline
    \end{tabular}
    \caption{List of Prior works on Convergence Rates. The notation $f(t) = O(g(t))$ is defined as the existence of a positive constant $c > 0$ such that $f(t) \leq c \, g(t)$ for sufficiently large $t > 0$. Such definition along with other measure of convergence rates such as $f(t) = \Theta(g(t))$ and $f(t) = o(g(t))$ can be found in \cref{tab:1}.}
    \label{tab:thesis:sgd:intro:rates}
\end{table}

\subsection{Acknowledgment}

I would like to thank Prof. Sunder Sethuraman, for discussions and support throughout the writing process of this manuscript.

\pagebreak
\section{Mathematical Model and Main Results}

\subsection{Notation and Assumptions}

We list notations, assumptions, results, and main mathematical arguments that we use in our work for convenience of the reader.

\begin{table}[h!]
    \centering
    \begin{tabular}{|c||p{12cm}|}
        \hline
        Symbol & Description
        \\
        \hline\hline
        $\N_0 = \N \cup \{0\}$ & The set of non-negative integers.
        \\
        $\R_+$ & The open interval $(0, \infty)$, i.e. the set of strictly positive real numbers.
        \\
        $\lesssim$ & $a \lesssim b$ means that there exists $c > 0$ such that $a \leq c \cdot b$.
        \\
        $\| \cdot \|$ & The standard Euclidean norm in $\R^d$.
        \\
        $f(t) = o(g(t))$ & $f(t)/g(t) \rightarrow 0$ as $t \rightarrow \infty$.
        \\
        $f(t) = O(g(t))$ & $\exists\, c > 0$ and $T > 0$ such that $f(t) \leq c \,g(t)$ for all $t > T$.
        \\
        $f(t) = \Theta(g(t))$ & $f(t) = O(g(t))$ and $g(t) = O(f(t))$.
        \\
        $(\Omega, \Fa, \Prob)$ & The probability space for the paired process $(Z_t, w_t)$ in (\ref{eq:shb}).
        \\
        $\Fa_{t}$ & The filtration $\{\Fa_{t} : t \in \N\}$ where $\Fa_t = \sigma\{w_s : s \leq t\}$.
        \\
        $\Fa_{\infty}$ & The $\sigma$-algebra $\bigcup_{t \in \N} \Fa_t$.
        \\
        $m\Fa_{t}$ & The set of measurable functions w.r.t $\Fa_{t}$ for $t \in \N \cup \{\infty\}$.
        \\
        $E$ a.s. & The event $E$ occurs almost surely, i.e. $\Prob(E) = 1$.
        \\
        $E$ h.p. & The event $E$ occurs with high probability, i.e. $\exists \delta > 0,~ \Prob(E) \geq 1 - \delta$.
        \\
        $\ell(z, w)$ & The estimator $\ell : \Zb \times \Wb \rightarrow [0, \infty)$ of the objective function.
        \\
        $Z_t$ & An i.i.d sequence of random vectors in $\Zb \subset \R^N$ that makes up the stochastic gradient estimator $\nabla \ell(Z_t, w_t)$.
        \\
        $\rho$ & The density function for $Z_t$.
        \\
        $F$ & The objective function $F : \Wb \rightarrow [0, \infty)$ defined by $F(w) = \Exp_{\rho}[\ell(Z_t, w)]$.
        \\
        $\delta m_t$ & The injected noise in stochastic algorithms, defined in (\ref{eq:sgd2}).
        \\
        \hline
    \end{tabular}
    \caption{Summary of Notations.}
    \label{tab:1}
\end{table}

\begin{assumption}
    Let $\Zb \subset \R^n$, $\Wb \subset \R^d$ and $\rho$ be a probability density on $\Zb$. The function $\ell : \Zb \times \Wb \rightarrow \R_+$ satisfies:
    \begin{enumerate}
        \item $\ell(z, \cdot)$ is \textbf{\textit{$(\gamma, L)$-smooth}}, i.e. $\ell(z, \cdot)$ is differentiable and $\nabla \ell(z, \cdot)$ is $(\gamma, L)$-H\"{o}lder:
        \begin{equation*}
            \exists \gamma \in (0, 1]~ \text{such that}~ \|\nabla \ell(z, u) - \nabla \ell(z, v)\| \leq L \,\|u - v\|^{\gamma} ~~~\forall\, u, v \in \Wb,\, z \in \Zb.
        \end{equation*}
        \item $\nabla \ell(\cdot, w) \in L^{1}(\Zb, \Ba, \rho\,dz)$ for all $w \in \Wb$ where $\Ba$ is the Borel $\sigma$-algebra in $\R^N$.
        \item $F(w) = \Exp_{\rho}[\ell(Z, w)]$ is bounded below by $F_* = \inf_{w \in \Wb} F(w) > -\infty$ and a global minimizer $w_* \in \Wb$ exists, i.e. $F(w_*) = F_*$.
    \end{enumerate}
    \label{a:loss}
\end{assumption}

The assumption that $\nabla \ell$ is only $\gamma$-H\"{o}lder is used in \cite[assumption 1]{lei2018}, whose application is also mentioned in \cite{lei2018}. This is a weakening of the $\nabla \ell$ being Lipschitz used in \cite[assumption 1]{liuyuan2024}. However, thanks to a slight modification in the proof of Garrigos and Gower \cite[lemma 2.25]{garrigos2024} (see proposition \ref{p:gamma_smooth}), any function $f$ that is $(\gamma, L)$-smooth still enjoys the property
\begin{equation}
    f(y) \leq f(x) + \InnerProd{\nabla f(x)}{y - x} + \frac{L}{1 + \gamma} \|y - x\|^{1+\gamma}
    ~~~\,\forall x, y \in \R^d.
    \label{eq:gamma_smooth}
\end{equation}

Now, observe that $\Exp_{\rho}[\nabla \ell(Z, w)]$ exists due to the second property in assumption \ref{a:loss}. Such a property can be satisfied for example when $\Zb$ is bounded or if $Z$ is Gaussian. Moreover, assuming such property is natural; in the gradient descent (GD) scheme
\begin{equation}
    w_{t+1} - w_t = -\alpha_t \, \frac{1}{n} \sum_{k=1}^{n} \nabla\ell(Z_k, w_t)
    \label{eq:sgd_det}
\end{equation}
such an assumption implies convergence almost surely of the empirical average on the right hand side to $\Exp_{\rho}[\nabla\ell(Z, w_t)] = \nabla \Exp_{\rho}[\ell(Z, w_t)] = \nabla F(w_t)$ as the number of data points $n \uparrow \infty$. This suggests that (\ref{eq:sgd_det}) approximates the deterministic GD \cite[algorithm 3.2]{garrigos2024}. In addition, there are important consequences. The first consequence is due to proposition \ref{p:F_inherits_L}: $F$ inherits the properties in assumption \ref{a:loss} from the function $\ell(z, \cdot)$. Secondly, if the joint distribution of $(Z_t, w_t)$ is the product measure $\rho \,dz \otimes \mu_t$, then
\begin{equation}
    \begin{aligned}
        \int \Indicator{\{ w_t \in A \}} \,\CondExp{\nabla\ell(Z_t, w_t)}{\Fa_t} \,d\Prob &= \int_A \int_{\Zb} \nabla\ell(z,w) \,\rho(z) \,dz \,\mu_t(dw)\\
        &= \int_A \Exp_{\rho}[\nabla\ell(Z, w)] \,\mu_t(dw) = \int_A \nabla\Exp_{\rho}[\ell(Z, w)] \,\mu_t(dw)\\
        &= \int \Indicator{\{ w_t \in A \}} \,\nabla F(w_t) \,d\Prob.
    \end{aligned}
    \label{eq:grad_F_estimator}
\end{equation}
Equation (\ref{eq:grad_F_estimator}) implies $\CondExp{\nabla\ell(Z_t, w_t)}{\Fa_t} = \nabla F(w_t)$ and that the estimator $\nabla\ell(Z, w_t)$ is unbiased (w.r.t $\rho$). That $\nabla F(w_t)$ can be accessed through the estimator $\nabla\ell(Z, w_t)$ is also assumed in \cite[the sentence right above assumption 4]{liuyuan2024}.

\begin{assumption}
    Using the same notation as in assumption \ref{a:loss}, there exists $A,B,C \in \R_+$ such that $\CondExp{\|\nabla \ell(Z_t, w_t)\|^{1+\gamma}}{\Fa_t} \leq A(F(w_t) - F_*) + B\,\|\nabla F(w_t)\|^{1+\gamma} + C$ for all $t > 0$ a.s.-$\Prob$.
    \label{a:ABC}
\end{assumption}

The above assumption is called the \textbf{\textit{ABC condition}}, originally proposed by Khaled and Richt\'{a}rik in \cite{khaled2023}, and is used by Liu and Yuan in \cite[assumption 4]{liuyuan2024} when $F$ is $(1, L)$-smooth, i.e. $\gamma = 1$. It is said to be ``the weakest assumption" for analysis of SGD in the non-convex setting \cite[remark 1]{liuyuan2024}. By lemma \ref{l:lsg_14}, the first bounding term is a consequence of $(\gamma, L)$-smoothness and convexity:
\begin{equation*}
    \begin{aligned}
        \int \Indicator{\{w_t \in A\}} \,&\CondExp{\|\nabla\ell(Z_t, w_t)\|^{1+\gamma}}{\Fa_t} \,d\Prob = \int_A \Exp_{\rho}\|\nabla\ell(z,w)\|^{1+\gamma} \,\mu_t(dw)
        \\
        &\leq \int_A \bigg\{ c_1(\beta,\gamma) (F(w) - F_*) + c_2(\beta,\gamma) + c_3(\beta,\gamma)\, \Exp_{\rho}[\|\nabla\ell(Z, w_*)\|^{1+\gamma}] \bigg\} \,\mu_t(dw).
    \end{aligned}
\end{equation*}
Note that the last two terms in the integrand are constants and that the existence of a global minimizer $w_* \in \Wb$ is required (satisfied by assumption \ref{a:loss}). When $\gamma = 1$ and the gradient estimator is unbiased, i.e. $\Exp_{\rho}[\nabla \ell(Z,w)] = F(w)$, the variance is
\begin{equation*}
    \begin{aligned}
        \Exp_{\rho}\|\nabla \ell(Z, w) - \nabla F(w)\|^2  &\leq \Exp_{\rho}\|\nabla \ell(Z, w)\|^2 - 2 \, \Exp_{\rho}[\InnerProd{\nabla \ell(Z, w)}{\nabla F(w)}] + \|\nabla F(w)\|^2
        \\
        &= \Exp_{\rho}\|\nabla \ell(Z, w)\|^2 - \|\nabla F(w)\|^2.
    \end{aligned}
\end{equation*}
Therefore, if the variance is bounded by $\sigma^{2} < \infty$, then $\Exp_{\rho}\|\nabla \ell(Z, w)\|^2 \leq \|\nabla F(w)\|^2 + \sigma^{2}$. This justifies the second term in the ABC assumption. For $\gamma < 1$, assuming that the gradient estimator has variance bounded by $\sigma^2$, we end up with
\begin{equation*}
    \Exp_{\rho}\|\nabla \ell(Z_t, w_t)\|^{1+\gamma} \leq 2^{1+\gamma} (\|\nabla F(w_t)\|^{1+\gamma} + \Exp_{\rho}\|\delta m_t\|^{1+\gamma}) \leq B \,\|\nabla F(w_t)\|^{1+\gamma} + C\, \sigma^{1+\gamma}.
\end{equation*}
In fact, with a simple estimate $a^{1+\gamma} \leq (a + 1)^{1+\gamma} \leq (a + 1)^2 \leq 2 a^2 + 2$, setting $a = \|\nabla F(w_t)\|$ gives
\begin{equation}
    \CondExp{ \|\nabla \ell(Z_t, w_t)\|^{1+\gamma} }{\Fa_{t}} \leq 2 B \,\|\nabla F(w_t)\|^{2} + (C\, \sigma^{1+\gamma} + 2) ~~~\text{a.s.-}\Prob.
    \label{eq:ABC:grad_F}
\end{equation}
which has the form of the ABC condition in \cite[assumption 4]{liuyuan2024}. Further discussions on the ABC condition can be found in \cite{liuyuan2024} and \cite{khaled2023}.

\begin{assumption}
    $\sup_{z \in \Zb} \ell(z, w_*) < \infty$.
    \label{a:ell_cond}
\end{assumption}

Assumption \ref{a:ell_cond} is also used in \cite{lei2018}. Such an assumption can be satisfied when $\Zb$ is bounded. See discussions in \cite{lei2018}. Note that by lemma \ref{l:lsg_sb_ell}.\ref{i:l:lsg_sb_ell:dmt}, if $\ell(z, \cdot)$ is convex and assumption \ref{a:loss} holds, then the `noise' $\|\delta m_t\|^2$ defined in (\ref{eq:sgd2}) is bounded by a scalar multiple of $\|w_t - w_*\|^{2\gamma} + 1$. This assumption is used to produce probability estimates using concentration inequalities for the convergence rate of SHB iterates.

\subsection{Almost Sure Convergence Rate}
\label{s:main_results:1}

Using assumptions \ref{a:loss} and \ref{a:ABC}, one may estimate the iterates $F(w_t) - F_*$. Indeed, under $\gamma$-H\"{o}lder smoothness assumption, we have $F(y) - F(x) \leq \InnerProd{\nabla F(x)}{y-x} + (L/(1+\gamma)) \, \|y - x\|^{1+\gamma}$ for all $x, y \in \R^d$. If one substitutes $y = w_{t+1}$ and $x = w_t$, then for the SGD case, we get an estimate for the iterates of $F(w_t)$ almost surely:
\begin{equation}
    F(w_{t+1}) - F(w_t) \leq -\alpha_t \, \|\nabla F(w_t)\|^2 + \alpha_t \, \InnerProd{\nabla F(w_t)}{\delta m_t} + \frac{L}{1 + \gamma} \, \alpha_t^{1+\gamma} \| \nabla \ell(Z_t, w_t) \|^{1+\gamma}.
    \label{eq:sgd:intro:dF}
\end{equation}

A natural approach to deal with \eqref{eq:sgd:intro:dF} is to apply $\CondExp{\cdot}{\Fa_t}$. Upon such application, the martingale difference term vanishes and the negative term $-\alpha_t \|\nabla F(w_t)\|^2$ can be dropped. So the iterates of $\CondExp{F(w_{t+1}) - F_*}{\Fa_t}$ looks like
\begin{equation}
    \CondExp{F(w_{t+1}) - F_*}{\Fa_t} \leq [F(w_t) - F_*] + \frac{L}{1 + \gamma} \, \alpha_t^{1+\gamma} \, \CondExp{\|\nabla \ell(Z_t, w_t)\|^{1+\gamma}}{\Fa_t}.
    \label{eq:thesis:sgd:intro:idea:RS}
\end{equation}
Under \cref{a:ABC}, $Y_t := F(w_t) - F_*$ is a non-negative almost super-martingale for which an almost sure convergence theorem (a.k.a. \terminology{Robbins-Siegmund}) exists \cite{robbins1971}. In \cite{liuyuan2024}, an upper bound for $t^r Y_t$ is constructed and then Robbins-Siegmund theorem is used. The Robbins-Siegmund theorem is proved by constructing a super-martingale and then using Doob's martingale convergence theorem to conclude the convergence \cite{robbins1971}. Our approach is slightly different. While we use the same estimate as in \cite[lemma 19]{liuyuan2024} for the upper bound of $t^r Y_t$, we proved the convergence of $t^r Y_t$ using Gronwall's inequality and Doob's Martingale convergence theorem. Gronwall gives uniform upper bounds to $\sum_{t=1}^{\infty} \alpha_t \, \|\nabla F(w_t)\|^2$, $\sum_{t=1}^{\infty} \alpha_t \, Y_t$, and $t^r Y_t$. Then, one can estimate $\CondExp{t^r Y_t}{\Fa_s} - s^r Y_s$ from above by an absolutely convergent series for some rate $r > 0$ where $\varliminf_{s \uparrow \infty} s^r Y_s = 0$ if $Y_t > 0$. Due to the requirement $Y_t > 0$, our result holds only up to the stopping time $\tau := \inf \{ t > 0 : F(w_t) = F_*\}$.

\begin{proposition}
    Let $X_t, Y_t, Z_t$ be non-negative for all $t \in \N_0$ and $a_t > 0$ be such that
    \begin{enumerate}
        \item $a_t, Z_t \in \ell^1(\N)$.
        \item $Y_t \leq (1 + a_{t-1})Y_{t-1} - X_{t-1} + Z_{t-1}$.
    \end{enumerate}
    Then $Y_t \in \ell^{\infty}(\N)$ and $X_t \in \ell^1(\N)$.
    \label{p:weak_RS}
\end{proposition}

\begin{lemma}
    Let $Y_t, Z_t$ be a sequence of positive and non-negative random variables respectively, adapted to a filtration $\Fa_t$ and $c_1, c_2 \in \R_+$ be constants. Let $\alpha_t = \Theta(t^{-p})$ where $p \in (\frac{1}{1+\gamma}, 1)$ and $\gamma \in (0, 1]$. If both of the following holds
    \begin{enumerate}
        \item For a fixed $s > 0$, $\CondExp{Y_{t+1}}{\Fa_s} \leq \CondExp{Y_t}{\Fa_s} (1 + c_1 \, \alpha_t^{1+\gamma}) + \alpha_t^{1+\gamma} (\CondExp{Z_t}{\Fa_s} + c_2)$ for all $t \geq s$.
        \item $\sum_{t=1}^{\infty} \alpha_t \, \Exp[Y_t] < \infty$ and $\sum_{t=1}^{\infty} \alpha_t \, \Exp[Z_t] < \infty$.\label{i:l:weak_RS_rate:expectation}
    \end{enumerate}
    then $\lim_{t \rightarrow \infty} t^{\min(1-p,p(1+\gamma)-1)-\eps} \, Y_t = 0$.
    \label{l:weak_RS_rate}
\end{lemma}

\begin{remark}
    Compared to \cite[lemma 19]{liuyuan2024} and \cite[theorem 1]{robbins1971}, we require $Y_t > 0$ and item \ref{i:l:weak_RS_rate:expectation} where the expectation of the sum is required to be finite. The application of \cref{l:weak_RS_rate} will be for convex cost $F$ and $Y_t = F(w_t) - F_*$. Due to the requirement $Y_t > 0$, instead of stating general convergence rate for $F(w_t) - F_*$, we only make statements about the convergence rate of $F(w_{\tau \wedge t}) - F_*$ where $\tau := \inf \{ t > 0 : F(w_t) = F_* \}$.
\end{remark}

\begin{theorem}
    \label{t:risk_o}
    Let $w_t$ be the iterates in (\ref{eq:shb}). With the notation in table \ref{tab:1}, if $\alpha_t = O(t^{-p})$ with $p \in (\frac{1}{1+\gamma}, 1)$ and both assumption \ref{a:loss} and \ref{a:ABC} hold, then the SHB converges almost surely at rate
    \begin{enumerate}
        \item $\ds \min_{0 \leq s \leq t} \|\nabla F(w_s)\|^2 = o(t^{p-1})$ a.s.-$\Prob$.
        \item If in addition, $\ell(z, \cdot)$ is convex for all $z \in \Zb$, then
        \begin{enumerate}
            \item $\ds \min_{0 \leq s \leq t}  F(w_s) - F_* = o(t^{p-1})$ as $t \uparrow \infty$.
            \item Let $\tau := \inf \{t > 0 : F(w_t) = F_*\}$. We have
            \begin{equation*}
                F(w_{\tau \wedge t}) - F_* = o(t^{r_\gamma \max(p-1,1-(1+\gamma)p)+\eps}) ~~\text{where}~~ r_\gamma = \begin{cases}
                    \frac{2\gamma}{1+\gamma} & \beta \in (0, 1)
                    \\
                    1 & \beta = 0
                \end{cases}.
            \end{equation*}
        \end{enumerate}
    \end{enumerate}
\end{theorem}

\begin{remark}
    The slowdown factor $r_\gamma$ due to the smoothness parameter $\gamma$ appears only with momentum $\beta > 0$ and $\gamma < 1$. The $r_\gamma$ factor is counter-intuitive since momentum is supposed to make convergence faster when the gradient is small. Examining (\ref{eq:c:F_shb:1}), the factor $\frac{\beta}{1-\beta}$ appears as a scalar multiplier to the `slower' converging terms, which slows down the overall rate for $\beta > 0$.
\end{remark}

\subsection{Convergence Rate with High Probability}
\label{s:main_results:2}

Observe that the inner product term on the right hand side of \eqref{eq:sgd:intro:dF} is a martingale difference. If we sum both sides of \eqref{eq:sgd:intro:dF} from $t$ to $T-1$, then we get
\begin{equation}
    \begin{aligned}
        F(w_T) - F_* &\leq [F(w_t) - F_*] + \sum_{s=t}^{T-1} \alpha_s \bigg( \InnerProd{\nabla F(w_s)}{\delta m_s} - \|\nabla F(w_s)\|^2 \bigg)
        \\
        &+ \frac{L}{1+\gamma} \sum_{s=t}^{T-1} \alpha_s^{1+\gamma} \, \|\nabla\ell(Z_s, w_s)\|^{1+\gamma}.
    \end{aligned}
    \label{eq:sgd:intro:ftoc_F}
\end{equation}
The term $M_T := \sum_{s=t}^{T-1} \alpha_s \InnerProd{\nabla F(w_s)}{\delta m_s}$ is now a martingale which can be estimated by a term smaller than $\sum_{s=t}^{T-1} \alpha_s \|\nabla F(w_s)\|^2$ with high probability using martingale concentration inequalities of \terminology{Bernstein} and \terminology{Azuma-Hoeffding}. These concentration inequalities require strong uniform bound on the martingale difference, which assuming that $\nabla F$ is Lipschitz, can be estimated by Cauchy-Schwarz $|\InnerProd{\nabla F(w_s)}{\delta m_s}| \leq \|w_s - w_*\| \cdot \|\delta m_s\|$. Therefore, such an approach requires estimating $\|w_s - w_*\|$ and $\|\delta m_s\|$ either almost surely or with high probability (\cref{l:conc_1} and \cref{c:max_noise}). To produce uniform estimates for these quantities to apply concentration inequalities, we require that the step size be `polynomially bounded', i.e. there exists $c_1, c_2 > 0$ such that $c_1 t^{-p} \leq \alpha_t \leq c_2 t^{-p}$ for all $t > 0$ (cf. \cref{tab:1}).

There is a nice intuition associated to this approach. Compared to a deterministic algorithm where the error $F(w_t) - F_*$ decreases at every iteration \cite[thm 3.4, lemma 2.28]{garrigos2024}, the effectiveness of a stochastic algorithm is shown over several iterations $[t, T]$. That is, one must wait until inefficiencies due to noise `averages out' over several iterations. Therefore, estimating $F(w_t) - F_*$ as a sum over an interval as in \eqref{eq:sgd:intro:ftoc_F} can be meaningful. We remark that with existing assumptions (\cref{a:loss} and \cref{a:ABC}), it is difficult to show directly that $F(w_t) - F_*$ decreases almost surely by conventional means (e.g. Borel-Cantelli), which is a key to achieve an optimal convergence rate of $O(t^{-1/2})$ shown in \cite{agarwal2012} (refer to \cref{r:thesis:sgd:fast_conv}).

\begin{theorem}
    Let $z_t,w_t,v_t$ be the SHB iterates as in (\ref{eq:shb}) and (\ref{eq:shb3}) and $K_0(L,\gamma,\beta),\, K_5(L,\gamma,\beta)$ be positive constants defined in (\ref{eq:const}) and proposition \ref{p:wt_hp}. If $\ell(z, \cdot)$ is convex for all $z \in \Zb$, assumption \ref{a:loss} and assumption \ref{a:ell_cond} both hold with $\gamma = 1$, $\beta \in [0, 1)$, and $\alpha_t$ satisfies
    \begin{enumerate}
        \item $\ds \alpha_0 \leq \min \left( 1, \frac{1}{2 \sqrt{K_5}} , K_0 \right)$.
        \item $\alpha_{t+1} \leq \alpha_{t}$ for all $t \geq 0$.
        \item $\alpha_t = \Theta(t^{-p})$ for $p \in (\frac{1}{2}, 1)$.
    \end{enumerate}
    then
    \begin{equation*}
        \Prob \left[ F(w_{T+1}) - F_* = O \left( T^{\max(p - 1,-2p+1)} \left( \log \frac{T}{\delta}\right)^2 \right) \right] \geq 1-\delta.
    \end{equation*}
    \label{t:risk_hp}
\end{theorem}

\begin{remark}
    We remark that it is required to have $c_1 t^{-p} \leq \alpha_t \leq c_2 t^{-p}$ for all $t > 0$, which is more stringent than the asymptotic $\Theta$-bound or $O$-bound since we need to estimate ratios of the form $\alpha_{t+s}/\alpha_t$ in (\ref{eq:p:sum_vt:sum_xt}) and $(\sum \alpha_t)^{-1}$ in theorem \ref{t:hp}. This result is consistent with theorem \ref{t:risk_o} and \cite[theorem 13]{liuyuan2024}.
\end{remark}

\section{Almost Sure Convergence Rate}
\subsection{Stochastic Gradient Descent}

\begin{proposition}
    If assumption \ref{a:loss}, assumption \ref{a:ABC} all hold and $\alpha_t \in \ell^{1+\gamma}(\N)$ satisfies $\alpha_{t} \leq \min\{1, \frac{1}{\sqrt[\gamma]{LB}}\}$ for all $t \geq 0$, then the SGD algorithm (\ref{eq:shb}) with $\beta = 0$ satisfies
    \begin{enumerate}
        \item $\sup_{t \geq 0} \Exp[F(w_t) - F_*] < \infty$ and $\sum_{t=1}^{\infty} \alpha_{t} \,\Exp\|\nabla F(w_t)\|^2 < \infty$.
        \item If in addition $\alpha_t = O(t^{-p})$ where  $p \in (\frac{1}{1+\gamma}, 1)$, then $\ds \min_{1 \leq s \leq t} \|\nabla F(w_s)\|^2 = o(t^{p-1})$ a.s.-$\Prob$.
        \item If in addition $\ell(z, \cdot)$ is convex for all $z \in \Zb$, then $\sum_{t=1}^{\infty} \alpha_t \, \Exp[F(w_t) - F_*] < \infty$.
    \end{enumerate}
    \label{p:F_wt}
\end{proposition}
\begin{proof}
    Due to $(\gamma, L)$-smoothness of $F$ and assumption \ref{a:ABC}, by proposition \ref{p:gamma_smooth},
    \begin{equation}
        \begin{aligned}
            \CondExp{F(w_{t}) &- F_* - (F(w_{t-1}) - F_*)}{\Fa_{t-1}} \leq \CondExp{F(w_{t}) - F(w_{t-1})}{\Fa_{t-1}}
            \\
            &\leq \InnerProd{ \nabla F(w_{t-1}) }{ \CondExp{w_{t} - w_{t-1}}{\Fa_{t-1}} } + \frac{L}{1+\gamma}\,\CondExp{\|w_{t} - w_{t-1}\|^{1+\gamma}}{\Fa_{t-1}}
            \\
            &\leq -\alpha_{t-1} \|\nabla F(w_{t-1})\|^2 + \frac{L}{1+\gamma} \,\alpha_{t-1}^{1+\gamma} \,\left( A\,(F(w_{t-1}) - F_*) + B\,\|\nabla F(w_{t-1})\|^{1+\gamma} + C \right)
        \end{aligned}
        \label{eq:p:F_wt:iterates}
    \end{equation}
    By (\ref{eq:ABC:grad_F}), we get
    \begin{equation}
        \begin{aligned}
            \CondExp{F(w_{t}) - F_* - &(F(w_{t-1}) - F_*)}{\Fa_{t-1}} \leq -\alpha_{t-1} \|\nabla F(w_{t-1})\|^2
            \\
            &+ \frac{L}{1+\gamma} \,\alpha_{t-1}^{1+\gamma} \,\left( A\,(F(w_{t-1}) - F_*) + 2B\,\|\nabla F(w_{t-1})\|^{2} + 2 + C \right).
        \end{aligned}
        \label{eq:p:F_wt:est}
    \end{equation}
    If $Y_t = \Exp[F(w_t) - F_*]$ and $LB \alpha_{t}^{\gamma} \leq 1$, then
    \begin{equation}
        Y_t \leq \left( 1 + \frac{L A}{1 + \gamma} \,\alpha_{t-1}^{1+\gamma} \right) \,Y_{t-1} - \left( \frac{\gamma}{1+\gamma} \right) \,\alpha_{t-1}\,\Exp\|\nabla F(w_{t-1})\|^2 + \frac{L\,(C + 2)}{1 + \gamma} \,\alpha_{t-1}^{1+\gamma}.
        \label{eq:p:F_wt:E_F_wt}
    \end{equation}
    Applying proposition \ref{p:weak_RS} yields the first claim. For the second claim, note that $\sum \alpha_t \,\|\nabla F(w_t)\|^2 < \infty$ almost surely since its expectation is finite (due to the first claim). This implies
    \begin{equation}
        \sum_{t=1}^{\infty} \left( \alpha_t \cdot \min_{1 \leq s \leq t} \|\nabla F(w_s)\|^2 \right) \leq \sum_{t=1}^{\infty} \alpha_t \, \|\nabla F(w_t)\|^2 < \infty.
        \label{eq:p:F_wt:min_grad_F}
    \end{equation}
    Since in addition $\min_{1 \leq s \leq t} \|\nabla F(w_s)\|^2$ is monotonically decreasing, by \cite[theorem 3.3.1]{knopp1956}, we have $t^{1-p} \,\min_{1\leq s \leq t}\|\nabla F(w_s)\|^2 \rightarrow 0$ almost surely.

    For the third claim, let $W_t = \Exp\|w_t - w_*\|^2$. By the convexity of $F$, we obtain the estimate
    \begin{equation*}
        \begin{aligned}
            W_{t+1} &= W_t + 2 \, \Exp[\InnerProd{w_{t+1} - w_t}{w_t - w_*}] + \Exp\|w_{t+1} - w_t\|^2
            \\
            &= W_t - 2 \alpha_t \, \Exp\InnerProd{\nabla F(w_t)}{w_t - w_*} + \alpha_t^2 \Exp\| \nabla \ell(Z_t, w_t)\|^2
            \\
            &\leq W_t - 2 \alpha_t \, \Exp[F(w_t) - F_*] + \alpha_t^2 \left( A \, \Exp[F(w_t) - F_*] + B \, \alpha_t \, \Exp\|\nabla F(w_t)\|^2 + C \right).
        \end{aligned}
    \end{equation*}
    Since $\alpha_t \in \ell^2(\N)$, by proposition \ref{p:F_wt}, the third term in the above estimate is summable. We may apply proposition \ref{p:weak_RS} and conclude $\sum \alpha_t \, \Exp[F(w_t) - F_*] < \infty$.
\end{proof}

\begin{remark}
    Note that the first two results in proposition \ref{p:F_wt} applies to non-convex cost function, comparable with \cite[thm 5.12]{garrigos2024} by Garrigos and Gower for $\gamma = 1$, which asserts that for a fixed end time $T \geq 1$, under a constant step size of the order $O(T^{-1/2})$, we have $\min_{0 \leq t < T} \Exp[\|\nabla F(w_t)\|^2] = O(T^{-1/2})$. More generally, our result is consistent with \cite[theorem 6.2]{liuyuan2024}.
\end{remark}

\begin{remark}
    From proposition \ref{p:F_wt}, we may conclude that $\Prob(\sum \alpha_t \, \|\nabla F(w_t)\|^2 < \infty) = 1$ and if $F$ is convex, $\Prob[\sum \alpha_t \, (F(w_t) - F_*)] = 1$. Intuitively, this suggests that if $\sum \alpha_t = \infty$, then $\|\nabla F(w_t)\|^2 \rightarrow 0$ almost surely (resp. for the convex case $F(w_t) - F_* \rightarrow 0$ almost surely). Moreover, if $F(w_t) - F_*$ decreases eventually, then we obtain the optimal convergence in \cite{agarwal2012}.
    \label{r:thesis:sgd:fast_conv}
\end{remark}

\begin{corollary}
    \label{cor:grad_F}
    If all assumptions in proposition \ref{p:F_wt} holds and $\alpha_t \notin \ell^1(\N)$, then
    \begin{enumerate}
        \item $\|\nabla F(w_t)\| \rightarrow 0$ almost surely.
        \item If $\ell(z, \cdot)$ is convex for all $z \in \Zb$, then
        \begin{enumerate}
            \item $F(w_t) - F_* \rightarrow 0$ almost surely.
            \item Let $\tau$ be defined as in \cref{t:risk_o}. If $\alpha_t = \Theta(t^{-p})$ for $p \in (\frac{1}{1+\gamma}, 1)$, then almost surely,
            \begin{equation*}
                F(w_{\tau \wedge t}) - F_* = o(t^{\max(p-1,1-(1+\gamma)p)+\eps}).
            \end{equation*}
        \end{enumerate}
    \end{enumerate}
\end{corollary}
\begin{proof}
    Following the same argument as in \cite[theorem 2]{orabona2020}, we get
    \begin{equation*}
        \begin{aligned}
            \bigg| \|\nabla F(w_t)\| - \|\nabla F(w_s)\| \bigg|^{1/\gamma} &\leq \|\nabla F(w_t) - \nabla F(w_s)\|^{1/\gamma}
            \leq
            L \sum_{i=s}^{t} \alpha_i \, \|\nabla F(w_i)\| + L \, \left\| \sum_{i=s}^{t} \alpha_i \, \delta m_i \right\|.
        \end{aligned}
    \end{equation*}
    Since proposition \ref{p:F_wt} asserts $\sum \alpha_t \, \|\nabla F(w_t)\|^2 < \infty$ almost surely, by lemma \ref{l:orabona}, it suffices to show that the second term on the right hand side has finite $L^2$ norm. Indeed, if this is the case, then $\sum_{i=s}^{t} \alpha_i \delta m_i < \infty$ almost surely. By proposition \ref{p:mt} and assumption \ref{a:ABC},
    \begin{equation*}
        \begin{aligned}
            \Exp \left\| \sum_{i=s}^{t} \alpha_i \, \delta m_i \right\|^2 = \sum_{i=s}^{t} \alpha_i^2 \, \Exp\|\delta m_i\|^2
            \leq
            \sum_{i=s}^{t} \alpha_i^2 \left( A \, \Exp[F(w_i) - F_*] + B \, \Exp\|\nabla F(w_i)\|^2 + C \right)
        \end{aligned}
    \end{equation*}
    The term on the right hand side is finite due to the first conclusion in proposition \ref{p:F_wt} and that $\alpha_t \in \ell^{1+\gamma} \subset \ell^2$. Therefore, the first claim follows.

    \textbf{Part 2.a.}
    For the convex case, we know that $F$ has no saddle points and all minima are global. Therefore, the set of critical points $\Cb = \{ w \in \Wb : \|\nabla F(w)\|^2 = 0 \}$ is the set of global minima. Therefore, $\|\nabla F(w_t)\|^2 \rightarrow 0$ implies $w_t \rightarrow w$ where $w \in \Cb$. By continuity of $F$, we get $F(w_t) \rightarrow F(w) = F_*$.

    \textbf{Part 2.b.}
    For a realization $\omega \in \Omega$, if $\tau(\omega) < \infty$, then $F(w_{\tau(\omega) \wedge t}(\omega)) = F_*$ for $t \geq \tau(\omega)$. The claim then holds trivially. If $\tau(\omega) = \infty$, then $F(w_{\tau(\omega) \wedge t}(\omega)) = F(w_t(\omega))$ for all $t > 0$. Defining $Y_t = F(w_t) - F_*$, applying $\CondExp{\cdot}{\Fa_s}$ on (\ref{eq:p:F_wt:est}), dropping the $-\alpha_t \, \|\nabla F(w_t)\|^2$ term, and applying lemma \ref{l:lsg_13} yields an inequality of the form
    \begin{equation*}
        \begin{aligned}
            \CondExp{Y_{t+1}}{\Fa_s} &\leq \left( 1+ c_0 \, \alpha_t^{1+\gamma} \right) \CondExp{Y_t}{\Fa_s} + c_0 \, \alpha_t^{1+\gamma} \bigg( c_1 \, \CondExp{Y_t}{\Fa_s} + c_2 \bigg)
            \\
            &= \left(1 + (c_0 + c_0 c_1) \alpha_t^{1+\gamma} \right) \, \CondExp{Y_t}{\Fa_s} + c_0 c_2 \, \alpha_t^{1+\gamma}
        \end{aligned}
    \end{equation*}
    where $c_0 = \frac{L}{1+\gamma}$, $c_1 = A + 4 BL$, and $c_2 = C + 2$. By proposition \ref{p:F_wt}, all conditions in lemma \ref{l:weak_RS_rate} are satisfied. Therefore, the claim holds.
\end{proof}

\subsection{Stochastic Heavy Ball}

The algorithm for SHB is given in (\ref{eq:shb}). From \cite[eq. (17) and (18)]{liuyuan2024}, by defining
\begin{equation}
    \begin{cases}
        v_{t} = w_{t} - w_{t-1} ~,~ z_{t} = w_{t} + \frac{\beta}{1-\beta} \,v_{t} ~~,~~ t \geq 1.
        \\
        w_1 = w_0.
    \end{cases}
    \label{eq:shb2}
\end{equation}
the SHB can be rewritten as one-step iterates
\begin{equation}
    \begin{cases}
        v_{t+1} = \beta\,v_t - \alpha_{t}\,\nabla\ell(Z_t, w_t) & t \geq 1.
        \\
        z_{t+1} = z_t - \ds\frac{\alpha_{t}}{1-\beta}\,\nabla\ell(Z_t, w_t) & t \geq 1.
        \\
        z_1 = w_1 ~,~ v_1 = 0.
    \end{cases}
    \label{eq:shb3}
\end{equation}

\begin{proposition}
    If assumption \ref{a:loss} and assumption \ref{a:ABC} holds, then there exists a constant $b_0 > 0$ such that if $\alpha_{t} \in \ell^{1+\gamma}$ and $\alpha_{t} \leq \min(1, b_0)$, then the SHB algorithm (\ref{eq:shb3}) corresponding to (\ref{eq:shb}) satisfies
    \begin{enumerate}
        \item $\sup_{t \geq 0} \Exp[F(z_t) - F_* + \|v_t\|^2] < \infty$ and $\sum_{t=1}^{\infty} \alpha_{t} \,\Exp\|\nabla F(z_t)\|^2 < \infty$.
        \item $\sum_{t=1}^{\infty} \Exp\|v_t\|^2 < \infty$ a.s.-$\Prob$.
        \item If $\alpha_t = O(t^{-p})$ for $p \in (\frac{1}{1+\gamma}, 1)$, then $\ds \min_{1 \leq s \leq t} \|\nabla F(w_s)\|^2 = o(t^{p-1})$ a.s.-$\Prob$.
        \item If $\ell(z, \cdot)$ is convex for all $z \in \Zb$, then $\sum_{t=1}^{\infty} \alpha_t  \left( \Exp[F(w_t) - F_*] + \Exp[F(z_t) - F_*] \right) < \infty$.
    \end{enumerate}
    \label{p:F_shb_wt}
\end{proposition}
\begin{proof}
    \textbf{Proof of claim 1.} Following \cite[eq. (21)-(22)]{liuyuan2024} in the proof of \cite[theorem 8]{liuyuan2024} and using the simple estimate $a^{1+\gamma} \leq 2a^2 + 2$ as in (\ref{eq:ABC:grad_F}), we have
    \begin{align}
        \|\nabla F(w_t)\|^2 &\leq 2 \|\nabla F(z_t)\|^2 + c_1 \,\|v_t\|^2
        \label{eq:p:F_shb_wt:grad_F_wt_zt}
        \\
        F(w_t) - F_* &\leq F(z_t) - F_* + \frac{\|\nabla F(z_t)\|^2}{2} + \left\{ c_2 + c_3 \right\} \,\|v_t\|^{2} + c_4
        \label{eq:p:F_shb_wt:F_wt_zt}
    \end{align}
    where
    \begin{gather*}
        c_1 := \frac{2 L^2 \beta^2}{(1-\beta)^2}
        ~~,~~
        c_2 := \frac{\beta^2}{2(1-\beta)^2}
        ~~,~~
        c_3 := \frac{L \beta^{1+\gamma}}{(1+\gamma)(1-\beta)^{1+\gamma}} ~~,~~
        c_4 := 2(c_3 + 1).
    \end{gather*}
    Observe that the left hand sides of (\ref{eq:p:F_shb_wt:grad_F_wt_zt}) and (\ref{eq:p:F_shb_wt:F_wt_zt}) makeup the bounding terms in assumption \ref{a:ABC}. This implies that the upper bound of $\CondExp{\nabla \ell(Z_t, w_t)}{\Fa_t}$ can be expressed entirely in terms of SHB variables $F(z_t) - F_*$, $\|\nabla F(z_t)\|^2$, and $\|v_t\|^2$. As a consequence, we may estimate the iterates of the SHB variables as follows
    \begin{equation}
        \begin{aligned}
            \CondExp{\|v_{t+1}\|^2}{\Fa_{t}} &\leq \beta^2\,(1 + \e_1)\, \|v_t\|^2 + \frac{1}{\e_1}\,\alpha_t^2\, \|\nabla F(w_t)\|^2
            \\
            &+ \alpha_t^2 \left( A \,[F(w_t) - F_*] + B\, \|\nabla F(w_t)\|^2 + C \right)
            \\\\
            &\leq \left( \beta^2\,(1 + \e_1) + \left\{ \left( \frac{1}{\e_1} + B \right) c_1 + A \,c_2 + A\,c_3 \right\}\, \alpha_t^2 \right) \|v_t\|^2
            \\
            &+ \left( \frac{2}{\e_1} + 2 B + \frac{A}{2} \right)\, \alpha_t^2 \,\|\nabla F(z_t)\|^2 + \alpha_t^2 \,\bigg( C + A \, c_4 \bigg)
            \\
            &+ A \,\alpha_t^2 \,[F(z_t) - F_*]
        \end{aligned}
        \label{eq:p:F_shb_wt:v_est}
    \end{equation}
    where $\e_1$ is an arbitrary constant. In addition, following \cite[eq.(20)]{liuyuan2024}, we get
    \begin{equation}
        \begin{aligned}
            \CondExp{F(z_{t+1})}{\Fa_{t}} &\leq F(z_t) - \left( \frac{1}{1-\beta} - \frac{\alpha_{t} \,c_5}{\e_2} \right) \alpha_t \,\|\nabla F(z_t)\|^2 + \e_2 \|v_t\|^2
            \\
            &+ c_6 \, \alpha_t^{1+\gamma} \,(A[F(w_t) - F_*] + B\|\nabla F(w_t)\|^2 + C)
            \\\\
            &\leq F(z_t) - \left( \frac{1}{1-\beta} - \frac{\alpha_{t} \,c_5}{\e_2} - \left( B\, c_6 + \frac{A}{2} \right) \,\alpha_t^{\gamma} \right) \alpha_t \,\|\nabla F(z_t)\|^2
            \\
            &+ \bigg( \e_2 + (c_1\,c_6\,B + A \,c_2 + A\,c_3)\,\alpha_t^{1+\gamma} \bigg)\, \|v_t\|^2
            \\
            &+ A\,c_6\,\alpha_{t}^{1+\gamma}\,[F(z_t) - F_*] + ( C + A\,c_4 ) \,\alpha_t^{1+\gamma}
        \end{aligned}
        \label{eq:p:F_shb_wt:F_est}
    \end{equation}
    where $\e_2$ is an arbitrary constant and
    \begin{gather*}
        c_5 := \frac{c_1(L,\beta)}{8(1-\beta)}
        ~,~
        c_6 :=  \frac{c_3(L,\beta,\gamma)}{\beta^{1+\gamma}}
    \end{gather*}
    Since $\beta < 1$, we take $\e_1, \e_2 \in \R_+$ such that $\beta^2(1 + \e_1) + \e_2 \leq 1$, e.g. set $\e_1 = \frac{1-\beta^2}{2\beta^2}$ and $\e_2 = \frac{1-\beta^2}{2}$. Now, adding (\ref{eq:p:F_shb_wt:v_est}) and (\ref{eq:p:F_shb_wt:F_est}) yields that there exists positive constants $b_0, b_1, b_2, b_3 \in \R_+$ such that if $\alpha_t \leq \min(1, b_0)$, then
    \begin{equation}
        \begin{aligned}
            \CondExp{F(z_{t+1}) - F_* + \|v_{t+1}\|^2}{\Fa_{t}} &\leq (1 + b_1\, \alpha_{t}^{1+\gamma}) (F(z_t) - F_* + \|v_t\|^2)
            \\
            &- b_2\, \alpha_t \,\|\nabla F(z_t)\|^2 + b_3 \,\alpha_t^{1+\gamma}.
        \end{aligned}
        \label{eq:p:F_shb_wt:F_est_final}
    \end{equation}
    Applying the expectation on both sides and proposition \ref{p:weak_RS} yields the first claim.

    \textbf{Proof of claim 2 and 3.} Observe that the iterates of $v_t$ satisfies
    \begin{equation}
        \begin{aligned}
            \Exp\|v_{t+1}\|^2 &= \beta^2 \, \Exp\|v_t\|^2 - 2 \beta \, \alpha_t \, \Exp[\InnerProd{\nabla F(w_t)}{v_t}] + \alpha_t^2 \, \Exp\|\nabla \ell(Z, w_t)\|^2
            \\
            &\leq \Exp\|v_t\|^2 - 2 \beta \, \alpha_t \, \Exp[\InnerProd{\nabla F(w_t)}{v_t}] + \alpha_t^2 \, \Exp\|\nabla \ell(Z, w_t)\|^2.
        \end{aligned}
        \label{eq:p:F_shb_wt:vt}
    \end{equation}
    Using assumption \ref{a:ABC}, the first claim, along with $\ell^{1+\gamma} \subset \ell^2$, (\ref{eq:p:F_shb_wt:grad_F_wt_zt}), and (\ref{eq:p:F_shb_wt:F_wt_zt}) yields that the last term in (\ref{eq:p:F_shb_wt:vt}) is summable:
    \begin{equation*}
        \sum_{t=1}^{\infty} \alpha_t^2 \,\bigg( \Exp\|\nabla F(z_t)\|^2 + \sup_{t \geq 0}\, \Exp\left[F(z_t) - F_* + \|v_t\|^2\right] + 1 \bigg) < \infty.
    \end{equation*}
    Therefore, by proposition \ref{p:weak_RS}, the middle term in (\ref{eq:p:F_shb_wt:vt}) is also summable: $\sum \alpha_t \Exp[\InnerProd{\nabla F(w_t)}{v_t}] < \infty$. Now we have that $\Exp\|v_{t+1}\|^2 - \beta^2 \Exp\|v_t\|^2$ corresponds to summable terms. Since $v_1 = 0$, we know $\sum_{t=1}^{\infty} \|v_t\|^2 = \sum_{t=1}^{\infty} \|v_{t+1}\|^2$. Therefore,
    \begin{equation*}
        (1-\beta^2) \sum_{t=1}^{\infty} \Exp\|v_t\|^2 = -2 \beta \sum_{t=1}^{\infty} \alpha_t \, \Exp[\InnerProd{\nabla F(w_t)}{v_t}] + \sum_{t=1}^{\infty} \alpha_t^2 \, \Exp\|\nabla \ell(Z_t, w_t)\|^2 < \infty.
    \end{equation*}
    This proves $\sum \Exp\|v_t\|^2 < \infty$, which implies $\sum_{t=1}^{\infty} \|v_t\|^2 < \infty$ a.s.-$\Prob$. The third claim follows from (\ref{eq:p:F_shb_wt:grad_F_wt_zt}) using similar argument as in the proof of the second result in proposition \ref{p:F_wt}.

    \textbf{Proof of claim 4.} Define $U_t = \Exp\|z_t - w_*\|^2$ and $K_{ABC} = C + (A + 2BL) \sup_{t \geq 0} \Exp[F(w_t) - F_*]$. By the convexity of $F$,
    \begin{equation*}
        \begin{aligned}
            U_{t+1} &= U_t -2 \alpha_t \, \Exp[\InnerProd{\nabla F(w_t)}{z_t - w_*}] + \alpha_t^2 \, \Exp\|\nabla \ell(Z_t, w_t)\|^2
            \\
            &= U_t - 2 \alpha_t \left( \Exp[\InnerProd{\nabla F(w_t)}{w_t - w_*}] + \frac{\beta}{1-\beta} \Exp[\InnerProd{\nabla F(w_t)}{v_t}]\right) + \alpha_t^2 \Exp\|\nabla \ell(Z_t, w_t)\|^2
            \\
            &\leq U_t - 2 \, \alpha_t \, \Exp[F(w_t) - F_*] + \left( \frac{2\beta}{1-\beta} \alpha_t (F(w_{t-1}) - F(w_t)) + \alpha_t^2 \, K_{ABC} \right)
            \\
            &\leq U_t - 2 \, \alpha_t \, \Exp[F(w_t) - F_*] + \left( \frac{2\beta}{1-\beta} \bigg( \alpha_{t-1} \Exp[F(w_{t-1})] - \alpha_t \Exp[F(w_t)] \bigg) + \alpha_t^2 \, K_{ABC} \right).
        \end{aligned}
    \end{equation*}
    Note that we have used convexity of $F$ and that $\alpha_t$ is decreasing in the second inequality. Note that terms in the parentheses are summable. Applying proposition \ref{p:weak_RS} yields that $\sum \alpha_t \, \Exp[F(w_t) - F_*] < \infty$ for which the fourth claim follows. Lastly, $(\gamma, L)$-smoothness gives
    \begin{equation*}
        \begin{aligned}
            F(z_t) - F_* &\leq F(w_t) - F_* + \InnerProd{\nabla F(w_t)}{z_t - w_t} + \frac{L}{1+\gamma} \|z_t - w_t\|^{1+\gamma}
            \\
            &\leq F(w_t) - F_* + \frac{\|\nabla F(w_t)\|^2}{2} + \frac{\beta^2}{2(1-\beta)^2} \|v_t\|^2 + \frac{L}{1+\gamma} \left(\frac{\beta}{1-\beta}\right)^{1+\gamma} \|v_t\|^{1+\gamma}.
        \end{aligned}
    \end{equation*}
    Multiplying both sides by $\alpha_t$ yields that the right hand side is summable.
\end{proof}

\begin{remark}
    Proposition \ref{p:F_shb_wt} shows that the SHB converges at rate similar to that of SGD. In particular, the SHB parameter $\beta$ in (\ref{eq:shb}) does not impact the almost sure convergence rate, consistent with Liu and Yuan's result \cite[theorem 8]{liuyuan2024}.
\end{remark}

\begin{corollary}
    \label{c:F_shb}
    If all the assumptions in proposition \ref{p:F_shb_wt} hold and $\alpha_t \notin \ell^1(\N)$, then
    \begin{enumerate}
        \item $\|\nabla F(w_t)\| \rightarrow 0$ almost surely.
        \item If $\ell(z, \cdot)$ is convex for all $z \in \Zb$, then
        \begin{enumerate}
            \item $F(w_t) - F_* \rightarrow 0$ almost surely.
            \item Let $\tau$ be defined as in \cref{t:risk_o}. If $\alpha_t = \Theta(t^{-p})$ for $p \in (\frac{1}{1+\gamma}, 1)$, then almost surely
            \begin{equation*}
                F(w_{\tau \wedge t}) - F_* = o(t^{r_\gamma \max(p-1,1-(1+\gamma)p)+\eps})
                ~~\text{where}~~
                r_\gamma = \begin{cases}
                    \frac{2\gamma}{1+\gamma} & \beta \in (0, 1)
                    \\
                    1 & \beta = 0
                \end{cases}.
            \end{equation*}
        \end{enumerate}
    \end{enumerate}
\end{corollary}
\begin{proof}
    The same argument as in corollary \ref{cor:grad_F} can be used. For a realization $\omega \in \Omega$, if $\tau(\omega) < \infty$, then $F(w_{\tau(\omega) \wedge t}(\omega)) = F_*$ for all $t \geq \tau(\omega)$ and the claim holds trivially. Otherwise, if $\tau(\omega) = \infty$, then $F(w_{\tau(\omega) \wedge t}(\omega)) = F(w_t(\omega))$ so that the argument in the next paragraph can be used.

    Both claims follow from (\ref{eq:p:F_shb_wt:grad_F_wt_zt}) and (\ref{eq:p:F_shb_wt:F_wt_zt}) along with the first and second conclusion in proposition \ref{p:F_shb_wt}. Let $Y_t = F(z_t) - F_* + \|v_t\|^2$. Applying $\CondExp{\cdot}{\Fa_s}$ on (\ref{eq:p:F_shb_wt:F_est_final}) and dropping the $-\alpha_t \|\nabla F(z_t)\|^2$ term yields
    \begin{equation*}
        \CondExp{Y_{t+1}}{\Fa_s} \leq (1 + b_1 \, \alpha_t^{1+\gamma}) \, \CondExp{Y_t}{\Fa_s} + b_3 \, \alpha_t^{1+\gamma}.
    \end{equation*}
    Proposition \ref{p:F_shb_wt} shows that all conditions in lemma \ref{l:weak_RS_rate} holds. Estimating $F(w_t) - F_*$ using $(\gamma, L)$-smoothness and applying lemma \ref{l:lsg_13} yields
    \begin{equation}
        \begin{aligned}
            F(w_t) - F_* &\leq F(z_t) - F_* + \frac{\beta}{1-\beta}\InnerProd{\nabla F(z_t)}{v_t} + \frac{L}{1+\gamma} \left( \frac{\beta}{1-\beta} \|v_t\| \right)^{1+\gamma}
            \\
            &\leq F(z_t) - F_* + \frac{\beta}{1-\beta} \left( \frac{\|\nabla F(z_t)\|^2 + \|v_t\|^2}{2} \right) + \frac{L}{1+\gamma} \left( \frac{\beta}{1-\beta} \|v_t\| \right)^{1+\gamma}
            \\
            &\leq F(z_t) - F_* + \frac{\beta}{1-\beta} \left( \frac{L^{1/\gamma}}{1+\gamma} \cdot (F(z_t) - F_*)^{\frac{2\gamma}{1+\gamma}} + \frac{\|v_t\|^2}{2} \right) + \frac{L}{1+\gamma} \left( \frac{\beta}{1-\beta} \|v_t\| \right)^{1+\gamma}.
        \end{aligned}
        \label{eq:c:F_shb:1}
    \end{equation}
    Applying lemma \ref{l:weak_RS_rate} yields $t^r (F(w_t) - F_*) \rightarrow 0$ if
    \begin{equation*}
        r = \begin{cases}
            \min(\frac{1+\gamma}{2}, \frac{2\gamma}{1+\gamma}) \cdot \min(1-p,p(1+\gamma)-1) & \beta \in (0, 1)
            \\
            \min(1-p,p(1+\gamma)-1) & \beta = 0.
        \end{cases}
    \end{equation*}
    This completes the proof.
\end{proof}

\section{Convergence Rate with High Probability for Convex Objective}

We will assume throughout this section that $\ell(z, \cdot)$ is convex for all $z \in \Zb$, which by proposition \ref{p:F_inherits_L}, makes $F$ also convex. It will be convenient to define once and for all the following constants and notations. Let $a_1$ and $a_2$ be the constants defined in lemma \ref{l:lsg_sb_ell}. Also, define the following
\begin{gather}
    \begin{aligned}
        k_0 &:= \frac{\beta}{1-\beta}
        ~~,~~
        k_1 := \frac{a_1(L, \gamma)}{(1-\beta)^2}
        ~~,~~
        k_2 := \frac{a_2(L, \gamma)}{(1-\beta)^2}
        ~~,~~
        k_3 := \frac{2}{1-\beta} \cdot \sup_{z \in \Zb} \ell(z, w_*).
        \\
        k_4 &:= k_0 + \beta \,k_0^2
        ~~,~~
        k_5 := k_2 + k_3 + a_2\,k_0^2.
        \\
        K_0 &:= \min \left( 1, \frac{1}{(a_1 \, k_0^2 + a_1 \, k_0 + a_1 + k_1)(1-\beta)} \right).
        \\
        \Delta z_t &= \|z_t - w_*\|^2 - \|z_{t-1} - w_*\|^2
        ~~,~~
        \Delta v_t = \|v_t\|^2 - \|v_{t-1}\|^2.
    \end{aligned}
    \label{eq:const}
\end{gather}

\subsection{Distance from a Minimizer with High Probability}

\begin{lemma}
    Let $z_t,w_t,v_t$ be the SHB iterates as in (\ref{eq:shb}) and (\ref{eq:shb3}). With the notation in (\ref{eq:const}), if $\ell(z, \cdot)$ is convex for all $z \in \Zb$, assumption \ref{a:loss} and assumption \ref{a:ell_cond} both hold, and $\alpha_t$ satisfies
    \begin{enumerate}
        \item $\alpha_0 \leq K_0$.
        \item $\alpha_{t+1} \leq \alpha_{t}$ for all $t \geq 0$.
        \item $\alpha_t \in \ell^2(\N)$.
    \end{enumerate}
    then there exists constants $K_1(L,\beta,\gamma,w_0) > 0$ and $K_2(L,\beta,\gamma,w_0) > 0$ such that for all $t > 0$,
    \begin{enumerate}
        \item $\ds \max \left\{ \|w_t - w_*\|^2, \|v_t\|^2, \|z_t\|^2 \right\} \leq K_1\, \sum_{s=0}^{t-1} \alpha_s$ a.s.-$\Prob$.
        \item $\sum_{s=1}^{t-1} \alpha_s^2\,\ell(Z_s, w_s) \leq K_2$ a.s.-$\Prob$.
    \end{enumerate}
    \label{l:conc_1}
\end{lemma}
\begin{proof}
    Consider (\ref{eq:shb3}) for the iterates of $v_t$. Take the squared-norm. Applying lemma \ref{l:lsg_sb_ell}.\ref{i:l:lsg_sb_ell:nabla_ell_ell} yields
    \begin{equation}
        \begin{aligned}
            \|v_{t+1}\|^2 - \|v_t\|^2 &\leq \|v_{t+1}\|^2 - \beta^2\,\|v_t\|^2
            \\
            &= -2\,\alpha_t\,\beta\,\InnerProd{\nabla \ell(Z_t, w_t)}{v_t} + \alpha_t^2 \,\|\nabla \ell(Z_t, w_t)\|^2
            \\
            &\leq -2\,\alpha_t \, \beta \, \InnerProd{\nabla \ell(Z_t, w_t)}{v_t}  + a_1\,\alpha_t^2 \,\ell(Z_t, w_t) + a_2\,\alpha_t^2
        \end{aligned}
        \label{eq:l:conc_1:dv}
    \end{equation}
    Using the notation in (\ref{eq:const}), convexity and lemma \ref{l:lsg_sb_ell}.\ref{i:l:lsg_sb_ell:nabla_ell_ell} yields
    \begin{equation}
        \begin{aligned}
            \Delta z_{t+1} &= -2\,\alpha_t\, \InnerProd{\nabla\ell(Z_t, w_t)}{\frac{ \beta v_t + w_t - w_*}{1-\beta} } + \frac{1}{(1-\beta)^2} \,\alpha_t^2 \,\|\nabla \ell(Z_t, w_t)\|^2
            \\
            &\leq -2\,\alpha_t\,\InnerProd{\nabla \ell(Z_t, w_t)}{\frac{\beta v_t}{1-\beta}} + \frac{2}{1-\beta} \,\alpha_t \bigg( \ell(Z_t, w_*) - \ell(Z_t, w_t)\bigg) + k_1\,\alpha_t^2 \,\ell(Z_t, w_t) + k_2\,\alpha_t^2
            \\
            &\leq -2\,\alpha_t\,\InnerProd{\nabla \ell(Z_t, w_t)}{\frac{\beta v_t}{1-\beta}} + \alpha_t \,(k_3 + k_2\,\alpha_t) - \alpha_t\,\ell(Z_t, w_t) \bigg( \frac{2}{1-\beta} - \alpha_t \,k_1 \bigg).
        \end{aligned}
        \label{eq:l:conc_1:dz}
    \end{equation}
    Now, let $A_t = \sum_{s=1}^{t-1} \alpha_s$. Note that $v_1 = 0$ and $z_1 = w_0$. Multiplying (\ref{eq:l:conc_1:dv}) by $k_0^2$, adding (\ref{eq:l:conc_1:dz}), and summing from $s=1$ to $s=t-1$ yields
    \begin{equation}
        \begin{aligned}
            \|z_t - w_*\|^2 + k_0^2\, \| v_t \|^2 &\leq
            \|w_0 - w_*\|^2  + k_5\, A_t
            \\
            &- 2\beta \left(\frac{1}{1-\beta} + k_0^2 \right) \sum_{s=1}^{t-1} \alpha_s\,\InnerProd{\nabla \ell(Z_s, w_s)}{v_s}
            \\
            &- \sum_{s=1}^{t-1} \alpha_s\,\ell(Z_s, w_s) \bigg( \frac{2}{1-\beta} - \alpha_s \,(k_1 + a_1 k_0^2) \bigg).
        \end{aligned}
        \label{eq:l:conc_1:zv}
    \end{equation}
    Let $X_t$ be the last term in (\ref{eq:l:conc_1:zv}) and $Y_t = 2 k_0 \InnerProd{w_0 - w_*}{v_t}$. Using (\ref{eq:l:conc_1:zv}) and $z_1 = w_0$, we get the estimate
    \begin{equation}
        \begin{aligned}
            \|w_t &- w_*\|^2 = \|z_t - w_*\|^2 + k_0^2\, \| v_t \|^2 - 2 k_0 \, \InnerProd{z_t - w_*}{v_t}
            \\
            &= \|z_t - w_*\|^2 + k_0^2\, \| v_t \|^2 - 2 k_0 \InnerProd{z_1 + w_*}{v_t} - 2 k_0 \sum_{s=1}^{t-1} \InnerProd{ z_{s+1} - z_s }{v_s}
            \\
            &= \|z_t - w_*\|^2 + k_0^2\, \| v_t \|^2 - 2 k_0 \InnerProd{w_0 - w_*}{v_t} + \frac{2 k_0}{1-\beta} \sum_{s=1}^{t-1} \alpha_s \,\InnerProd{ \nabla\ell(Z_s, w_s) }{v_s}
            \\
            &\leq \|w_0 - w_*\|^2 + k_5\,A_t + X_t - Y_t + 2 \, \beta \, k_0 \sum_{s=1}^{t-1} \alpha_s\,\InnerProd{ \nabla\ell(Z_s, w_s) }{v_s}.
        \end{aligned}
        \label{eq:l:conc_1:wt}
    \end{equation}
    Now, recalling $v_1 = 0$ and rearranging (\ref{eq:l:conc_1:dv}) and summing from $s=1$ to $s=t-1$ yields
    \begin{equation*}
        \begin{aligned}
            2\beta\,\sum_{s=1}^{t-1} \alpha_s \InnerProd{\nabla\ell(Z_s,w_s)}{v_s} &\leq - \|v_t\|^2 + a_1\, \sum_{s=1}^{t-1} \alpha_s^2 \,\ell(Z_s, w_s) + a_2  \sum_{s=1}^{t-1} \alpha_s^2.
        \end{aligned}
    \end{equation*}
    Combining this with (\ref{eq:l:conc_1:wt}) yields that the middle term $a_1 \sum \alpha_s^2 \ell(Z_s,w_s)$ can be collected with $X_t$ so that
    \begin{equation}
        \begin{aligned}
            \|w_t - w_*\|^2 &\leq \|w_0 - w_*\|^2 + k_5 \,A_t - Y_t - 2 k_0 \|v_t\|^2 - k_0 a_2  \sum_{s=1}^{t-1} \alpha_s^2
            \\
            &- \sum_{s=1}^{t-1} \alpha_s\,\ell(Z_s, w_s) \bigg( \frac{2}{1-\beta} - \alpha_s \,(k_1 + a_1 k_0^2 + a_1 k_0) \bigg).
        \end{aligned}
        \label{eq:l:conc_1:wt2}
    \end{equation}
    Requiring $\alpha_t$ to be small enough, we may drop the last two terms on the right hand side of (\ref{eq:l:conc_1:wt2}). Now, using the inequality $2|\InnerProd{w_0 - w_*}{v_t}| \leq \e \|w_0 - w_*\|^2 + \frac{1}{\e}\|v_t\|^2$ and setting $\e = 1/2$ yields
    \begin{equation}
        \begin{aligned}
            \|w_t - w_*\|^2 &\leq \|w_0 - w_*\|^2 + k_5 \,A_t - Y_t - 2 k_0 \|v_t\|^2
            \\
            &\leq \|w_0 - w_*\|^2 + k_5 \,A_t + 2 k_0 |\InnerProd{w_0 - w_*}{v_t}| - 2 k_0 \|v_t\|^2
            \\
            &\leq \|w_0 - w_*\|^2 + k_5 \,A_t +  \frac{k_0}{2} \|w_0 - w_*\|^2.
        \end{aligned}
        \label{eq:l:conc_1:wt3}
    \end{equation}
    Since we define $\alpha_0 > 0$, we have $\|w_t - w_*\|^2 \lesssim \sum_{s=0}^{t-1} \alpha_s$ as desired. Since $v_t = w_t - w_{t-1}$, we know that $\|v_t\|^2$ has the same order as $\|w_t - w_*\|^2$. Similarly, since $z_t = w_t + k_0 \,v_t$, it follows that $\|z_t\|^2$ also has the same order as $\|w_t - w_*\|^2$.

    Now, define $\Delta v_{t+1}^{\beta} = \beta^2 \|v_t\|^2 - \|v_{t+1}\|^2$. Observe that
    \begin{equation*}
        \begin{aligned}
            \Delta z_{t+1} - \Delta v_{t+1}^{\beta} &\leq (k_3 + k_2 \,\alpha_t) \alpha_t - \alpha_t\,\ell(Z_t,w_t)\left(\frac{2}{1-\beta} - \alpha_t \,k_1 \right) - \alpha_t^2 \|\nabla\ell(Z_t,w_t)\|^2
            \\
            &\leq k_5\, \alpha_t - \alpha_t\,\ell(Z_t,w_t)\left(\frac{2}{1-\beta} - \alpha_t \,k_1 \right) + \alpha_t^2 \|\nabla\ell(Z_t,w_t)\|^2
            \\
            &\leq (k_5 + a_2 \alpha_t)\, \alpha_t - \alpha_t\,\ell(Z_t,w_t)\left(\frac{2}{1-\beta} - \alpha_t \,(k_1 + a_1) \right).
        \end{aligned}
    \end{equation*}
    If $\alpha_t$ is small enough and is decreasing, rearranging the above by moving the last term to the left hand side and multiplying by $\alpha_t$ yields
    \begin{equation}
        \begin{aligned}
            \left( \frac{1}{1-\beta} \right) \alpha_t^2 \,\ell(Z_t,w_t) &\leq (k_5 + a_2\,\alpha_t)\,\alpha_t^2 - \Delta z_{t+1} + \Delta v_{t+1}^{\beta}
            \\
            &= (k_3 + a_2 \,\alpha_t)\,\alpha_t^2 - \alpha_t \Delta z_{t+1} + \alpha_t \beta^2 \|v_t\|^2 - \alpha_t \|v_{t+1}\|^2
            \\
            &\leq (k_3 + a_2)\,\alpha_t^2 + \alpha_t \|z_t\|^2 - \alpha_{t+1} \|z_{t+1}\|^2 + \alpha_t \|v_t\|^2 - \alpha_{t+1}\|v_{t+1}\|^2.
        \end{aligned}
        \label{eq:l:conc_1:a_loss}
    \end{equation}
    Since $\alpha_t \in \ell^2$, summing the above and making use of the first claim completes the proof.
\end{proof}

\begin{remark}
    Note that the $\beta$-parameter does not affect the long term fluctuation of the iterates as can be seen from the conclusion of lemma \ref{l:conc_1} and from setting $\beta = 0$ directly in estimates (\ref{eq:l:conc_1:wt3}) and (\ref{eq:l:conc_1:a_loss}). As such, lemma \ref{l:conc_1} holds if $w_t$ is the SGD iterates.
\end{remark}


In the sequel, we will require estimating constants of the form $\alpha_t \sum_{s=1}^{t} \alpha_s$. Under the assumption $\alpha_t \in \ell^{1+\gamma}(\N)$ for $\gamma \in (0, 1]$, we have $\alpha_t \sum_{s=0}^{t-1} \alpha_s \leq \alpha_t^{1-\gamma} \sum_{s=0}^{t-1} \alpha_s^{1+\gamma} < \infty$. Moreover, such a sum converges to zero for $\gamma \in (0, 1)$. If $\gamma = 1$, $\alpha_0 = K_0$, $\alpha_t \leq K_0 \, t^{-p}$ for $t > 0$, then
\begin{equation}
    \alpha_t \sum_{s=0}^{t-1} \alpha_s \leq K_0^2 \, t^{-p} \left( 1 + \sum_{s=1}^{t-1} s^{-p} \right) \lesssim K_0^2 \, \max(t^{-p} , (t-1)^{-2p+1}) \leq K_0^2 \leq K_0 ~~~,~~~ \forall t > 0~~~.
    \label{eq:stepsize}
\end{equation}

In addition, from here onwards, we will estimate probabilities using Azuma-Hoeffding and Bernstein's inequality following the work by Lei, Shi, Guo in \cite{lei2018}.

\begin{proposition}
    With the notation in lemma \ref{l:conc_1}, if all assumptions in lemma \ref{l:conc_1} holds and in addition $\alpha_t \leq K_0 t^{-p}$ for $p \in (\frac{1}{2}, 1)$, then there exists positive constants $K_3, K_4, K_5 \in \R_+$ depending only on $L,\gamma,\beta$ such that for any $t > 0$ and $\delta \in (0, 1)$,
    \begin{equation*}
        \Prob \left( \|z_{t} - w_*\|^2 + \left( \frac{\beta}{1-\beta} \right)^2\,\|v_{t}\|^2 \leq K_3 + K_4 \log \frac{1}{\delta} + K_5 \sum_{s=1}^{t-1} \alpha_s^2\,\|w_s - w_*\|^2 \right) \geq 1 - \delta.
    \end{equation*}
    \label{p:wt_hp}
\end{proposition}
\begin{proof}
    We use the notation of constants in (\ref{eq:const}). Using the convexity of $F$, summing (\ref{eq:l:conc_1:dv}) and (\ref{eq:l:conc_1:dz}) yields
    \begin{equation}
        \begin{aligned}
            \Delta z_{t+1} + k_0^2 \,\Delta v_{t+1} &\leq -2 \,\alpha_t \,\beta \left( \frac{1}{1-\beta} + k_0^2 \right) \InnerProd{\nabla \ell(Z_t, w_t)}{v_t} -  \frac{2\alpha_t}{1-\beta} \InnerProd{\nabla \ell(Z_t, w_t)}{w_t - w_*} + \varphi_t
            \\
            &= -2 \,\alpha_t \, k_4\, \InnerProd{\nabla F(w_t)}{w_t - w_{t-1}} - \frac{2\alpha_t}{1-\beta} \, \InnerProd{\nabla F(w_t)}{w_t - w_*} + \frac{2 \xi_t}{1-\beta} + \varphi_t
            \\
            &\leq 2 \, k_4 \, \alpha_t \big( F(w_{t-1}) - F(w_t) \big) + \frac{2\alpha_t}{1-\beta} \big( F(w_*) - F(w_t) \big) + \frac{2 \xi_t}{1-\beta} + \varphi_t
            \\
            &\leq 2 \, k_4 \, \big( \alpha_{t-1} F(w_{t-1}) - \alpha_t F(w_t) \big) + \frac{2}{1-\beta} \left( \alpha_t \, \big( F(w_*) - F(w_t) \big) + \xi_t \right) + \varphi_t
        \end{aligned}
        \label{eq:p:wt_hp:dz_dv}
    \end{equation}
    where $\xi_t$ and $\varphi_t$ are defined as $\Fa_t$-adapted processes
    \begin{equation}
        \begin{aligned}
            \xi_t &= (1-\beta) k_4 \alpha_t \InnerProd{\delta m_t}{v_t} + \alpha_t \InnerProd{\delta m_t}{w_t - w_*}.
            \\
            \varphi_t &= \alpha_t^2 \left( (a_1 k_0^2 + k_1) \,\ell(Z_t, w_t) + a_2 k_0^2 + k_2 \right).
        \end{aligned}
        \label{eq:p:wt_hp:xi_phi}
    \end{equation}
    By lemma \ref{l:conc_1}, we have
    \begin{equation}
        \sum_{t=1}^{\infty} \varphi_t \leq (a_1 k_0^2 + k_1) K_2 + (a_2 k_0^2 + k_2) < \infty ~~~\text{a.s.-}\Prob.
        \label{eq:p:wt_hp:phi}
    \end{equation}
    Also, we have $\CondExp{\xi_t}{\Fa_t} = 0$. We now show that we may bound $\sum_{k=1}^{t} \xi_k - \CondExp{\xi_k}{\Fa_{k}}$ using \cite[lemma 21]{lei2018}. By lemma \ref{l:lsg_sb_ell}.\ref{i:l:lsg_sb_ell:dmt}, lemma \ref{l:conc_1}, and (\ref{eq:stepsize}),
    \begin{equation}
        \begin{aligned}
            \xi_t - \CondExp{\xi_t}{\Fa_{t}} = \xi_t &\leq \alpha_t \|\delta m_t\| \left( k_4\, \|v_t\| + \|w_t - w_*\| \right)
            \\
            &\leq \alpha_t \sqrt{ 6 L^2 \|w_t - w_*\|^{\gamma} + a_4 } \left( k_4 + 1 \right) \sqrt{\sum_{s=0}^{t-1} \alpha_s}
            \\
            &\leq \left( k_4 + 1 \right) \,\left( 6 L^2 \, \sqrt{ \alpha_t^{1-\gamma} \left( \alpha_t \sum_{s=0}^{t-1} \alpha_s\right)^{\gamma + 1} }
            +
            a_4 \,\sqrt{\alpha_t^2 \sum_{s=0}^{t-1} \alpha_s} \right)
            \\
            &\leq \left( k_4 + 1 \right) \left( 6 L^2 \, \sqrt{K_0^{1+\gamma}} + a_4 \sqrt{K_0} \right).
        \end{aligned}
        \label{eq:p:wt_hp:mgdiff}
    \end{equation}
    As for the conditional variance, we apply Cauchy-Schwarz, lemma \ref{l:lsg_sb_ell}.\ref{i:l:lsg_sb_ell:dmt_F}, and lemma \ref{l:conc_1} to get
    \begin{equation*}
        \begin{aligned}
            \CondExp{ \xi_t^2 }{ \Fa_t } &\leq 2 \alpha_t^2 \left( k_4^2 \,\CondExp{ \InnerProd{\delta m_t}{v_t}^2 }{ \Fa_t } + \CondExp{ \InnerProd{\delta m_t}{w_t - w_*}^2 }{ \Fa_t } \right)
            \\
            &\leq 2 \left( k_4^2 + 1 \right) \alpha_t^2 \,\CondExp{\|\delta m_t\|^2}{\Fa_t} \bigg( \|v_t\|^2 + \|w_t - w_*\|^2 \bigg)
            \\
            &\leq 2 \left( k_4^2 + 1 \right) \alpha_t^2 \bigg( a_6 \, (F(w_t) - F_*) + a_7 \bigg) \left( \left( \sum_{s=0}^{t-1} \alpha_s \right) + \|w_t - w_*\|^2 \right).
        \end{aligned}
    \end{equation*}
    Recall, by lemma \ref{l:conc_1}, we get $\alpha_t \|w_t - w_*\|^2 \leq \alpha_t \sum_{s=0}^{t-1} \alpha_s \leq K_0$. Let $k_{11} = 2 \left( k_4^2 + 1 \right)$. There exists $c > 0$ such that
    \begin{equation}
        \begin{aligned}
            \CondExp{ \xi_t^2 }{ \Fa_t }
            &\leq
            2 k_{11} \, a_6 \, K_0 \, \alpha_t \, (F(w_t) - F_*) + k_{11} \, K_0 \, a_7 + k_{11} \, a_7 \, \alpha_t^2 \, \|w_t - w_* \|^2
            \\
            &\leq \underbrace{ k_{11} \max(2 K_0 \, a_6 \, a_7, K_0 a_7, a_7) }_{c} \left( 1 + \alpha_t \, (F(w_t) - F_*) + \alpha_t^2 \, \|w_t - w_* \|^2 \right).
        \end{aligned}
        \label{eq:p:wt_hp:xisq}
    \end{equation}
    Let $b > 0$ be the right hand side of (\ref{eq:p:wt_hp:mgdiff}). Applying \cite[lemma 21.b]{lei2018}, we get for a choice of tolerance level $\delta \in (0, 1)$ and all $q > 0$, the following holds with probability at least $1 - \delta$:
    \begin{equation}
        \begin{aligned}
            \sum_{s=1}^{t-1} \xi_s &\leq \frac{(e^{q} - q - 1)}{q} \cdot \frac{c}{b} \left( 1 + \sum_{s=1}^{t-1} \alpha_s \,(F(w_s) - F_*) + \sum_{s=1}^{t-1} \alpha_s^2 \, \|w_s - w_*\|^2 \right) + \frac{b}{q} \cdot \frac{\log(1/\delta)}{q}.
        \end{aligned}
        \label{eq:p:wt_hp:xi}
    \end{equation}

    Now, we want to `remove' the $\alpha_t \,(F(w_*) - F(w_t))$ term in (\ref{eq:p:wt_hp:dz_dv}) so that the upper bound of $\Delta z_{t+1} + k_0^2 \Delta v_{t+1}$ is of the order $O( 1 + \sum_{k=0}^{t-1} \alpha_k^2 \|w_k - w_*\|^2 )$. To do so, we choose $q > 0$ such that
    \begin{equation}
        \frac{e^q - q - 1}{q} \cdot \frac{c}{b} < 1
        \label{eq:p:wt_hp:k5}
    \end{equation}
    The above inequality is solvable since $(e^q - q - 1)/q \rightarrow 0$ as $q \rightarrow 0$. Moreover, $q$ is independent of $t$ and $\alpha_t$ since $b$ and $c$ are both independent of $t$ and $\alpha_t$. Therefore, all constants on the right hand side of (\ref{eq:p:wt_hp:xi}) are independent of $t$ and $\alpha_t$. Since $F(w_t) > F_*$, the term $\alpha_t (F_* - F(w_t))$ in (\ref{eq:p:wt_hp:dz_dv}) dominates over that in (\ref{eq:p:wt_hp:xi}). Using the estimate in (\ref{eq:p:wt_hp:xi}) and summing (\ref{eq:p:wt_hp:dz_dv}) from $1$ to $t-1$ yields the desired result.
\end{proof}

\begin{remark}
    Observe that the constant $K_5$ is independent of the step sizes. This can be seen from (\ref{eq:p:wt_hp:k5}). The constants $b$ and $c$ on the right hand side of (\ref{eq:p:wt_hp:mgdiff}) and (\ref{eq:p:wt_hp:xisq}) respectively are independent of the step sizes $\alpha_t$.
\end{remark}

\begin{proposition}
    With the notation in proposition \ref{p:wt_hp}, if all assumptions in lemma \ref{p:wt_hp} holds and $\alpha_0^2 \leq \frac{1}{4 K_5}$, then for any $\delta \in (0, 1)$ there exists a constant $K_6(L,\gamma,\beta) > 0$ such that for any $T > 0$
    \begin{equation*}
        \Prob \left( \max_{1 \leq t \leq T} \left\{ \|w_t - w_*\|^2, \|z_t - w_*\|^2, \|v_t\|^2 \right\} \leq K_6 \log\left( \frac{T}{\delta} \right) \right) \geq 1 - \delta.
    \end{equation*}
    \label{p:wt_max}
\end{proposition}
\begin{proof}
    It suffices to prove the conclusion for $\|w_t - w_*\|^2$ since $\|z_t - w_*\|^2 \leq 2 \|w_t - w_*\|^2 + 2 k_0^2 \|v_t\|^2$ and $v_t = (w_t - w_*) - (w_{t-1} - w_*)$ gives the following inequalities
    \begin{equation*}
        \begin{aligned}
            \max_{1 \leq t \leq T} \bigg( \|v_t\|^2 , \|z_t - w_*\|^2 \bigg) \leq 10 \max(1, k_0^2) \max_{1 \leq t \leq T} \|w_t - w_*\|^2.
        \end{aligned}
    \end{equation*}

    As a consequence of proposition \ref{p:wt_hp} and $\|w_t - w_*\|^2 \leq 2 \|z_t - w_*\|^2 + 2 k_0^2 \|v_t\|^2$, we have
    \begin{equation}
        \Prob\left( \frac{ \|w_t - w_*\|^2 }{2} > K_3 + K_4 \log\left( \frac{1}{\delta} \right) + K_5 \sum_{s=1}^{t-1} \alpha_s^2 \,\|w_s - w_*\|^2 \right) < \delta ~~~,~~~ \forall t \geq 1.
        \label{eq:p:wt_max:p}
    \end{equation}
    Let $t_1 = \min \{ t \in [1, T] : K_5 \sum_{s=t}^{T} \alpha_s^2 \leq \frac{1}{4} \}$. Note that $t_1$ is well-defined since $\alpha_T^2 \leq \alpha_0^2 \leq \frac{1}{4 K_5}$. In particular, $t_1$ depends only on $L, \gamma, \beta$ since $K_5$ depends only on those parameters. Using lemma \ref{l:conc_1}, with probability greater than $1-\delta$, for any $t \geq 1$,
    \begin{equation*}
        \begin{aligned}
            \frac{\|w_t - w_*\|^2}{2} &\leq K_3 + K_4 \,\log\left(\frac{T}{\delta}\right) + K_5 \,\sum_{s=1}^{t_1} \alpha_s^2 \,\|w_s - w_*\|^2 + K_5\,\sum_{s=t_1}^{T} \alpha_s^2 \,\|w_s - w_*\|^2
            \\
            &\leq K_3 + K_4 \,\log\left(\frac{T}{\delta}\right) + K_5 \sum_{s=1}^{t_1} \alpha_s^2 \left( \sum_{k=1}^{s} \alpha_k \right) + \frac{1}{4} \,\max_{t_1 \leq t \leq T} \|w_t - w_*\|^2
            \\
            &\leq \left( K_3 + K_5 \, K_0\,\sum_{s=1}^{t_1} \alpha_s \right) + K_4 \,\log\left(\frac{T}{\delta}\right) + \frac{1}{4} \,\max_{1 \leq t \leq T} \|w_t - w_*\|^2.
        \end{aligned}
    \end{equation*}
    Since the right hand side is independent of
    $t$, we know $\frac{1}{2} \max_{1 \leq t \leq T} \|w_t - w_*\|^2$ is also bounded by the right hand side. Solving the inequality for $\max_{1 \leq t \leq T} \|w_t - w_*\|^2$ yields the desired result, where $K_6 = \max\{ 4 K_3 + t_1 K_0, 4 K_4, 10 k_0^2, 10 \}$ for all $t > 0$.
\end{proof}

\begin{remark}
    Although $K_6$ depends on the choice of step sizes $\{ \alpha_t : t \geq 0 \}$, in the sequel, we will content ourselves with the choice $\alpha_t$ uniformly bounded by ``polynomially" decaying steps $\alpha_* \, t^{-p} \leq \alpha_t \leq \alpha^* \, t^{-p}$ for all $t > 0$ (which is stated as an assumption in \cref{t:risk_hp}). Under such assumption, the dependency of $K_6$ on $\alpha$ can be removed. Indeed, the parameter $t_1$ influencing $K_6$ can be chosen such that $K_5 \alpha^* \sum_{t=t_1}^{\infty} t^{-2p} < \frac{1}{4}$ for $p \in (\frac{1}{2}, 1]$.
    \label{r:thesis:sgd:step_size_polynomial}
\end{remark}

\subsection{Error Rate with High Probability}

\begin{corollary}
    With the notation in proposition \ref{p:wt_max}, define the event $\Eb_t^\delta$ as
    \begin{equation*}
        \Eb_t^{\delta} := \left\{ \max \left(\|w_t - w_*\|^2, \|z_t - w_*\|^2, \|v_t\|^2 \right) \leq K_6 \, \log \frac{k T}{\delta} \right\}.
    \end{equation*}
    If all assumptions in proposition \ref{p:wt_max} are satisfied, then under the event $\Eb_t^\delta$, with the notation in proposition \ref{p:wt_max} and lemma \ref{l:lsg_sb_ell},
    \begin{enumerate}
        \item $\Eb_t^{\delta} \in \Fa_t \subset \Fa_T$ and $\ds\Prob(\Eb_t^\delta) \geq 1-\delta$.
        \item $F(w_t) \leq (F_* + \frac{L}{1+\gamma} \cdot K_6) \left(\log \frac{T}{\delta}\right)^{\gamma}$ for all $t \leq T$ and $T > e \delta$.
        \item $\|\delta m_t\|^2 \leq (6 L^2 K_6 + a_4) \left(\log \frac{T}{\delta} \right)^{\gamma}$ for all $t \leq T$ and $T > e \delta$.
        \item $\|\nabla \ell(Z_t, w_t)\|^2 \leq (2 L^2 K_6 + a_5) \left(\log \frac{T}{\delta} \right)^{\gamma}$ for all $t \leq T$ and $T > e \delta$.
    \end{enumerate}
    \label{c:max_noise}
\end{corollary}
\begin{proof}
    Under the assumptions of corollary \ref{c:max_noise}, $T > 0$ is fixed and $t \leq T$. Therefore,  $\Eb_k$ is in $\Fa_t$. The proof follows directly from proposition \ref{p:wt_max}, lemma \ref{l:lsg_sb_ell}.\ref{i:l:lsg_sb_ell:F_wt}, \ref{l:lsg_sb_ell}.\ref{i:l:lsg_sb_ell:dmt}, and \ref{l:lsg_sb_ell}.\ref{i:l:lsg_sb_ell:nabla_ell_w}.
\end{proof}

\begin{proposition}
    With the notation in lemma \ref{l:conc_1}, if all of the assumptions in proposition \ref{p:wt_max} holds, then there exists a constant  $K_7(L,\gamma,\beta, w_0, w_*) > 0$ such that for any $T > 0$ and $0 < \delta < \min \left( 1, 2T/e \right)$,
    \begin{equation*}
        \Prob \left( \sum_{t=1}^{T} \alpha_t \,(F(w_t) - F_*) \leq K_7 \left( \log \frac{2T}{\delta} \right)^{3/2} \right) \geq 1-\delta.
    \end{equation*}
    \label{p:sum_weighted_risk}
\end{proposition}
\begin{proof}
    Using the notation in (\ref{eq:const}), rearranging terms in (\ref{eq:p:wt_hp:dz_dv}) and using $F(w_*) = F_*$ gives us
    \begin{equation}
        \frac{2}{1-\beta} \,\alpha_t (F(w_t) - F_*) \leq 2 \, k_4 \, (\alpha_{t-1} F(w_{t-1}) - \alpha_t F(w_t)) - \Delta z_{t+1} - k_0^2 \Delta v_{t+1} + \frac{2 \, \xi_t}{1-\beta} + \varphi_t
        \label{eq:p:sum_weighted_risk:1}
    \end{equation}
    where $\xi_t$ and $\varphi_t$ are defined in (\ref{eq:p:wt_hp:xi_phi}). By lemma \ref{l:lsg_sb_ell}.\ref{i:l:lsg_sb_ell:dmt}, we have
    \begin{equation}
        \begin{aligned}
            \xi_t &= (1-\beta) k_4 \, \alpha_t \InnerProd{\delta m_t}{v_t} + \,\alpha_t \InnerProd{\delta m_t}{w_t - w_*}
            \\
            &\leq \left( k_4 + 1 \right) \, \alpha_t \,\|\delta m_t\| \, \left( \|v_t\| + \|w_t - w_*\| \right)
            \\
            &\leq \left( k_4 + 1 \right) \alpha_t \, \sqrt{6 L^2 \|w_t - w_*\|^{2\gamma} + a_4} \, ( \|v_t\| + \|w_t - w_*\|)
        \end{aligned}
        \label{eq:p:sum_weighted_risk:xi}
    \end{equation}
    Let $\Eb_t^{\delta/2}$ be the event in corollary \ref{c:max_noise}. For a fixed $\delta \in (0, 1)$ and $T \geq t$, define $c = 4 \sqrt{6 L^2 + a_4} \left( k_4 + 1 \right)$ and $\xi_t' = \xi_t \Indicator{\Eb_t^{\delta/2}}$. Since $2T/\delta > e$, we have $\|w_t - w_*\|^{2\gamma} \leq \log(2T/\delta) \leq (\log(2T/\delta))^2$. Therefore, by (\ref{eq:p:sum_weighted_risk:xi}), we get
    \begin{equation*}
        \begin{aligned}
            \xi_t' &\leq 2\sqrt{2} \left( k_4 + 1 \right) \sqrt{ 2(6 L^2 + a_4) \,\alpha_t^2 \cdot \left( \log \frac{2 T}{\delta} \right)^2 }.
        \end{aligned}
    \end{equation*}
    Applying \cite[lemma 21.a]{lei2018} to $\xi_t'$ yields
    \begin{equation}
        \begin{aligned}
            \sum_{t=1}^{T} \xi_t' \leq \left( c \log \frac{2T}{\delta} \right) \left( 2 \log \frac{1}{\delta} \sum_{t=1}^{T} \alpha_t^2 \right)^{\frac{1}{2}} \leq c \sqrt{2} \left( \log \frac{2T}{\delta} \right)^{\frac{3}{2}}.
        \end{aligned}
        \label{eq:p:sum_weighted_risk:xi2}
    \end{equation}
    Let $K' > 0$ be the constant on the right hand side of (\ref{eq:p:wt_hp:phi}). Applying proposition \ref{p:wt_max} and substituting (\ref{eq:p:sum_weighted_risk:xi2}) to (\ref{eq:p:sum_weighted_risk:1}) yields that the following occurs with probability at least $1 - \delta$:
    \begin{equation*}
        \sum_{t=1}^{T} \alpha_{t} (F(w_t) - F_*) \leq \frac{1-\beta}{2} \left( \alpha_0 F(w_0) + \|z_1 - w_*\|^2 + \|v_1\|^2 + \frac{2}{1-\beta} \cdot c \sqrt{2} \left(\log \frac{2T}{\delta}\right)^{\frac{3}{2}} + K' \right).
    \end{equation*}
    Since $z_1 = w_0$, $v_1 = 0$, and $2T/\delta > e$, we take $K_7 = 4 \max \left( \alpha_0 \, F(w_0), \|w_0 - w_*\|^2, c \sqrt{2}, K' \right)$.
\end{proof}

\begin{proposition}
    With the notation as in lemma \ref{l:conc_1}, if all assumptions in proposition \ref{p:wt_max} holds, then there exists a constant $K_8(L,\gamma,\beta) > 0$ such that for any $\delta \in (0, 1)$, $T > 1$, and $\tau \in [1,T]$,
    \begin{equation*}
        \Prob \left[ \sum_{t=\tau+1}^{T} \alpha_t \, \|v_t\|^2 \leq K_8 \cdot \max \left( \alpha_{\tau}, \sqrt{\sum_{t=\tau}^{T} \alpha_t^4}, \sum_{t=\tau}^{T} \alpha_t^3 \right) \left( \log \frac{2T}{\delta} \right)^2 \right] \geq 1 - \delta.
    \end{equation*}
    \label{p:sum_at_vt}
\end{proposition}
\begin{proof}
    Consider the recursion $\alpha_{t+1} \, \|v_{t+1}\|^2 \leq \beta^2 \alpha_t \, \|v_t\|^2 - 2 \alpha_t^2 \beta \InnerProd{\nabla \ell(Z_t, w_t)}{v_t} + \alpha_t^3 \|\nabla \ell(Z_t, w_t)\|^2$ obtained from the recursion of $\|v_t\|^2$. Since $v_1 = 0$, $F$ is convex and $\alpha_t$ is decreasing, summing both sides from $t=\tau$ to $t=T$ yields
    \begin{align*}
        (1-\beta^2) \sum_{t=\tau+1}^{T} \alpha_t \, \|v_t\|^2 &\leq \beta^2 \alpha_{\tau} \|v_{\tau}\|^2 - \alpha_{T+1} \|v_{T+1}\|^2 +  \sum_{t=\tau}^{T} -2 \beta \alpha_t^2 \bigg( \InnerProd{\nabla F(w_t)}{v_t} - \InnerProd{\delta m_t}{v_t} \bigg)
        \notag
        \\
        &+ \sum_{t=\tau}^{T} \alpha_t^3 \|\nabla \ell(Z_t, w_t)\|^2
        \notag
        \\
        \notag
        \\
        &\leq \beta^2 \alpha_{\tau} \|v_{\tau}\|^2 + 2 \beta (\alpha_{\tau} F(w_{\tau}) - \alpha_T F(w_T))
        \notag
        \\
        &+ \sum_{t=\tau}^{T} \bigg\{ 2\beta \alpha_t^2 \InnerProd{\delta m_t}{v_t} + \alpha_t^3 \|\nabla \ell(Z_t, w_t)\|^2 \bigg\}.
    \end{align*}
    By corollary \ref{c:max_noise}, there exists an event $\Eb_t^{\delta/2}$ such that
    \begin{equation}
        \sum_{t=\tau+1}^{T} \alpha_t \, \|v_t\|^2 \leq c_0 \cdot \left( \log \frac{2T}{\delta} \right)^2 \left(  \alpha_\tau + \sum_{t=\tau}^{T} \left\{\alpha_t^2 \InnerProd{\delta m_t}{v_t} + \alpha_t^3 \right\} \right).
        \label{eq:p:sum_at_vt:1}
    \end{equation}
    where
    \begin{equation*}
        c_0 = \frac{1}{1-\beta^2} \cdot \max \left( \beta^2, 2\beta \left(F_* + \frac{L K_6}{1+\gamma}\right), 2L^2 K_6 + a_5 \right).
    \end{equation*}

    It remains to estimate the martingale part in (\ref{eq:p:sum_at_vt:1}). Applying Cauchy-Schwarz and lemma \ref{l:lsg_sb_ell}.\ref{i:l:lsg_sb_ell:dmt} gives
    \begin{equation*}
        \alpha_t^2 \cdot |\InnerProd{\delta m_t}{v_t}| \cdot \Indicator{\Eb_t^{\delta/2}} \leq \alpha_t^2 \cdot \sqrt{6L^2 + a_4} \cdot \left( \log \frac{2T}{\delta} \right)^{\gamma}
    \end{equation*}
    Using Azuma-Hoeffding \cite[lemma 21.a]{lei2018}, there exists an event $\Eb \in \Fa_T$ such that $\Prob(\Eb) \geq 1-\delta$ and $\cap_{t=\tau}^{T} \Eb_t^{\delta/2} \subset \Eb$ for which the following occurs
    \begin{equation*}
        \sum_{t=\tau+1}^{T} \alpha_t \|v_t\|^2 \leq c_0 \cdot \left( \log \frac{2T}{\delta} \right)^2 \left(  \alpha_\tau + \sqrt{ 2 \sum_{t=\tau}^{T} \alpha_t^4} + \sum_{t=\tau}^{T} \alpha_t^3 \right).
    \end{equation*}
    This completes the proof.
\end{proof}

\begin{lemma}
    With the notation as in lemma \ref{l:conc_1}, if all assumptions in proposition \ref{p:wt_max} holds, then there exists a constant $K_9(L,\gamma,\beta) > 0$ such that for any $\delta \in (0, 1)$, $T > 1$, and $\tau \in [1,T]$,
    \begin{equation*}
        \Prob \left[ \sum_{t=\tau}^{T} \alpha_t \InnerProd{\delta m_t}{v_t} \leq K_9 \cdot \max\left(\alpha_{\tau-1}, \sqrt{\sum_{t=\tau-1}^{T} \alpha_t^4}, \sum_{t=\tau-1}^{T} \alpha_t^3 \right) \left( \log \frac{3T}{\delta} \right)^2 \right] \geq 1 - \delta.
    \end{equation*}
    \label{l:sum_dmt_vt}
\end{lemma}
\begin{proof}
    Let $\Eb_t^{\delta/3}$ be the event defined in corollary \ref{c:max_noise}. Using Bernstein's inequality \cite[lemma 21.b]{lei2018}, we can bound the martingale $\sum_{t=\tau}^{T} \alpha_t \,\InnerProd{\delta m_t}{v_t}$ from above. By corollary \ref{c:max_noise}, we have the estimates
    \begin{equation*}
        \alpha_t \cdot |\InnerProd{\delta m_t}{v_t}| \cdot \Indicator{\Eb_t^{\delta/2}} \leq \sqrt{6 L^2 K_6 + a_4} \cdot \alpha_\tau \cdot \left( \log \frac{3T}{\delta} \right)^{\gamma}
    \end{equation*}
    and
    \begin{equation*}
        \alpha_t^2 \cdot \CondExp{|\InnerProd{\delta m_t}{v_t}|^2 \cdot \Indicator{\Eb_t^{\delta/2}}}{\Fa_t} \leq (6L^2 K_6 + a_4) \cdot \alpha_\tau \cdot\alpha_t \cdot \|v_t\|^2 \cdot \left( \log \frac{3T}{\delta} \right)^{\gamma}.
    \end{equation*}
    Let $b = \sqrt{6L^2 K_6 + a_4}$. Choosing $q > 0$ such that $(e^q-q-1)/q = 1$, by Bernstein's inequality \cite[lemma 21.b]{lei2018}, the following occurs with probability at least $1-\delta/3$:
    \begin{equation}
        \sum_{t=\tau}^{T} \alpha_t \, \InnerProd{\delta m_t}{v_t} \leq b \left( \sum_{t=\tau}^{T} \alpha_t \, \|v_t\|^2 \right) + \frac{b}{q} \cdot \alpha_\tau \left(\log \frac{3}{\delta}\right).
        \label{eq:l:sum_dmt_vt:1}
    \end{equation}
    Let $\Eb' \in \Fa_T$ denote the event of (\ref{eq:l:sum_dmt_vt:1}) and $\Eb_T' \in \Fa_T$ be the event in the conclusion of proposition \ref{p:sum_at_vt} replacing $\delta $ by $2\delta/3$ instead. Note that $\Prob(\Eb' \cap \Eb_T') \geq 1-\delta$ and in the event $\Eb' \cap \Eb_T' \in \Fa_T$,
    \begin{equation*}
        \sum_{t=\tau}^{T} \alpha_t \,\InnerProd{\delta m_t}{v_t} \leq b \cdot \left( K_8 + \frac{1}{q} \right) \cdot \max\left(\alpha_{\tau-1}, \sqrt{\sum_{t=\tau-1}^{T} \alpha_t^4}, \sum_{t=\tau-1}^{T} \alpha_t^3 \right) \cdot \left( \log \frac{3T}{\delta} \right)^2.
    \end{equation*}
    This completes the proof.
\end{proof}

\begin{lemma}
    With the notation as in lemma \ref{l:conc_1}, if all assumptions in proposition \ref{p:wt_max} holds, then there exists a constant $K_{10}(L,\gamma,\beta) > 0$ such that for any $\delta \in (0, 1)$, $T > 1$, and $\tau \in [0,T]$,
    \begin{equation*}
        \Prob \left[ \sum_{t=\tau}^{T} \InnerProd{\nabla F(w_t)}{-\alpha_t \nabla \ell(Z_t, w_t)} \leq K_{10} \cdot \alpha_{\tau} \left( \log \frac{2T}{\delta} \right)^{1+\gamma} \right] \geq 1 - \delta.
    \end{equation*}
    \label{l:sum_gradF_mt}
\end{lemma}
\begin{proof}
    We may decompose $-\alpha_t \InnerProd{\nabla F(w_t)}{\nabla\ell(Z_t, w_t)} = -\alpha_t \|\nabla F(w_t)\|^2 + \psi_t$ where we define
    \begin{equation}
        \psi_t' = \psi_t \Indicator{\Eb_t^{\delta/2}} ~,~ \psi_t = \alpha_t \InnerProd{\nabla F(w_t)}{\delta m_t}.
        \label{eq:l:sum_gradF_mt:psi}
    \end{equation}
    By Cauchy-Schwarz, $\gamma$-H\"{o}lder property of $\nabla F$, lemma \ref{l:lsg_sb_ell}.\ref{i:l:lsg_sb_ell:dmt}, for all $t \in [\tau, T]$,
    \begin{equation}
        \begin{aligned}
            |\psi_t'| &\leq \Indicator{\Eb} \cdot \alpha_t \cdot \|\nabla F(w_t)\| \cdot \|\delta m_t\|
            \\
            &\leq L \Indicator{\Eb} \cdot \alpha_t  \sqrt{ 6 L^2 \|w_t - w_*\|^{4\gamma} + a_4 \|w_t - w_*\|^{2\gamma} }
            \\
            &\leq L \sqrt{K_6^{\gamma} (6 L^2 K_6^{\gamma} + a_4)} \cdot \alpha_{\tau} \cdot \left( \log \frac{2T}{\delta} \right)^{2\gamma}.
        \end{aligned}
        \label{eq:l:sum_gradF_mt:psi_L1}
    \end{equation}
    We apply Cauchy-Schwarz and lemma \ref{l:lsg_sb_ell}.\ref{i:l:lsg_sb_ell:dmt} to get
    \begin{equation}
        \begin{aligned}
            \CondExp{(\psi_t')^2}{\Fa_t} &\leq \alpha_t^2 \cdot \|\nabla F(w_t)\|^2 \cdot \CondExp{\|\delta m_t\|^2 \cdot \Indicator{\Eb_t^{\delta/2}}}{\Fa_{t}}
            \\
            &\leq \alpha_\tau \cdot \alpha_t \, \|\nabla F(w_t)\|^2 \cdot ( 6 L^2 \|w_t - w_*\|^{2\gamma} \Indicator{\Eb_t^{\delta/2}}  + a_4 )
            \\
            &\leq \alpha_{\tau} \cdot \alpha_t \, \|\nabla F(w_t)\|^2 \cdot (6 L^2 + a_4)
            \cdot \left( \log \frac{2T}{\delta} \right)^{\gamma}.
        \end{aligned}
        \label{eq:l:sum_gradF_mt:psi_L2}
    \end{equation}
    For shorthand notation, denote by $b = L \sqrt{ K_6^{\gamma} (6 L^2 K_6 + a_4) }$ and $c = 6 L^2 + a_4$ as the constants appearing on the right hand side of (\ref{eq:l:sum_gradF_mt:psi_L1}) and (\ref{eq:l:sum_gradF_mt:psi_L2}). By Bernstein's inequality \cite[lemma 21.b]{lei2018}, for any $q > 0$ and $\delta \in (0, 1)$, there exists an event $\Eb' \in \Fa_T$ with $\Prob(\Eb') \geq 1 - \frac{\delta}{2}$ where $\Eb'$ is the event
    \begin{equation}
        \sum_{t=\tau}^{T} \psi_t' \leq \frac{e^q-q-1}{q} \cdot \frac{c}{b} \left( \sum_{t=\tau}^{T} \alpha_t \, \|\nabla F(w_t)\|^2 \right) + \alpha_{\tau} \cdot \frac{b}{q} \left( \log \frac{2T}{\delta} \right)^{1+\gamma}.
        \label{eq:t:hp:sum_psi}
    \end{equation}
    Since $(e^q - q - 1)/q \rightarrow 0$ as $q \rightarrow 0$, may choose $q > 0$ such that
    \begin{equation*}
        \frac{e^q - q - 1}{q} \cdot \frac{c}{b} < 1.
    \end{equation*}
    Let $\Eb \in \Fa_T$ be the event in the conclusion of proposition \ref{p:wt_max}, except with $\delta$ replaced by $\delta/2$. Under the event $\Eb \cap \Eb'$, we know $\psi_t' = \psi_t$ and the following occurs with probability $\Prob(\Eb \cap \Eb') \geq 1-\delta$
    \begin{equation*}
        \begin{aligned}
            \sum_{t=\tau}^{T} \InnerProd{\nabla F(w_t)}{-\alpha_t &\nabla \ell(Z_t, w_t)} = \sum_{t=\tau}^{T} \left( -\alpha_t \, \|\nabla F(w_t)\|^2 + \psi_t' \right)
            \\
            &\leq - \left( \left(1 - \frac{e^q-q-1}{q} \cdot \frac{c}{b} \right) \sum_{t=\tau}^{T}  \alpha_t \, \|\nabla F(w_t)\|^2 \right) + \frac{b}{q} \cdot \alpha_\tau \left(\log \frac{2T}{\delta} \right)^{1+\gamma}
            \\
            &\leq \frac{b}{q} \cdot \alpha_{\tau} \left( \log \frac{2T}{\delta} \right)^{1+\gamma}.
        \end{aligned}
    \end{equation*}
    The proof is complete.
\end{proof}

Starting from this point, we will assume that the step-sizes are ``polynomially" bounded as in \cref{t:risk_hp}. Therefore, per \cref{r:thesis:sgd:step_size_polynomial}, the constants appearing in the sequel is independent of $\{\alpha_t : t \geq 0\}$.

\begin{proposition}
    With the notation in lemma \ref{l:conc_1}, if all assumptions of proposition \ref{p:wt_max} holds, $\beta \in [0, 1)$, and there exists $c_1, c_2 \in \R_+$ such that $c_1 \, t^{-p} \leq \alpha_t \leq c_2 \, t^{-p}$ where $p \in (\frac{1}{2}, 1)$, then
    there exists a constant $K_{11}(L,\beta,\gamma,w_0,\|\alpha_t\|_{\ell^2(\N)}) > 0$ such that for any $\delta \in (0, 1)$, $T > 1$, and $\tau \in [1,T]$,
    \begin{equation*}
        \Prob \left[ \sum_{t=\tau}^{T} \|v_{t+1}\|^2 \leq K_{11} \, \left( \log  \frac{3T}{\delta} \right)^2 \max \left( \tau^{-2p+1},\,  \beta^{2\tau} \right) \right] \geq 1 - \delta.
    \end{equation*}
    \label{p:sum_vt}
\end{proposition}
\begin{proof}
    By lemma \ref{l:sum_dmt_vt} and proposition \ref{p:wt_max}, there exists an event $\Eb_T' \in \Fa_T$ such that with probability $1-\delta$ both the conclusion in lemma \ref{l:sum_dmt_vt} and proposition \ref{p:wt_max} holds. We have the recursion $\|v_{t+1}\|^2 = \beta^2 \|v_t\|^2 - 2 \beta \alpha_t \InnerProd{\nabla \ell(Z_t, w_t)}{v_t} + \alpha_t^2 \, \|\nabla \ell(Z_t, w_t)\|^2$.

    \textbf{Case $\beta = 0$.} If $\beta = 0$, then $\|v_{t+1}\|^2 = \alpha_t^2 \|\nabla \ell(Z_t, w_t)\|^2$. Under the event $\Eb_T'$, then by corollary \ref{c:max_noise}, we have
    \begin{equation*}
        \sum_{t=\tau}^{T} \|v_{t+1}\|^2 = \sum_{t=\tau}^{T} \alpha_t^2 \, \|\nabla \ell(Z_t, w_t)\|^2 \lesssim \left( \log \frac{T}{\delta}\right)^{\gamma} \sum_{t=\tau}^{T} \alpha_t^2 \leq \left( \log \frac{T}{\delta}\right)^{\gamma} \tau^{-2p+1}.
    \end{equation*}

    \textbf{Case $\beta \in (0, 1)$.} The recursion of $\|v_t\|^2$ in the previous paragraph can be rewritten into an integral equation by proposition \ref{p:ode}:
    \begin{equation*}
        \begin{cases}
            \|v_{t+1}\|^2 = \sum_{s=0}^{t-1} \beta^{2s} X_{t-s}.
            \\
            X_t = -2 \beta \alpha_t \InnerProd{\nabla \ell(Z_t, w_t)}{v_t} + \alpha_t^2 \|\nabla \ell(Z_t, w_t)\|^2.
        \end{cases}
    \end{equation*}
    It follows that
    \begin{equation}
        \begin{aligned}
            \sum_{t=\tau}^{T} \|v_{t+1}\|^2 &= \sum_{t=\tau}^{T} \sum_{s=1}^{t} \beta^{2(s-1)} X_{t-s+1}
            \\
            &= \left( \sum_{s=1}^{\tau} \sum_{t=\tau}^{T} + \sum_{s=\tau}^{T} \sum_{t=s}^{T} \right) \beta^{2(s-1)} X_{t-s+1}
            \\
            &= \sum_{s=1}^{\tau} \beta^{2 (s-1)} \left( \sum_{t=\tau}^{T} X_{t-s+1} \right) + \beta^{2 (\tau-1)} \sum_{s=0}^{T-\tau} \beta^{2s} \left( \sum_{t=1}^{T-\tau-s+1} X_t \right).
        \end{aligned}
        \label{eq:p:sum_vt:0}
    \end{equation}
    We estimate the parentheses in the first term. By lemma \ref{l:lsg_sb_ell}.\ref{i:l:lsg_sb_ell:nabla_ell_w} and the convexity of $F$, we get
    \begin{equation}
        \begin{aligned}
            X_t &= -2 \beta \, \alpha_t \, \InnerProd{\nabla F(w_t)}{v_t} + 2 \, \alpha_t \, \beta \InnerProd{\delta m_t}{v_t} + \alpha_t^2 \, \|\nabla \ell(Z_t, w_t)\|^2
            \\
            &\leq 2 \beta \alpha_t \left( F(w_{t-1}) - F(w_t) + \InnerProd{\delta m_t}{v_t} \right) + \alpha_t^2 \, (2L^2 \|w_t - w_*\|^{2\gamma} + a_5)
            \\
            &\leq 2 \beta (\alpha_{t-1} F(w_{t-1}) - \alpha_t F(w_t)) + 2 \beta \alpha_t \InnerProd{\delta m_t}{v_t} + \alpha_t^2 \, (2L^2 \|w_t - w_*\|^{2\gamma} + a_5).
        \end{aligned}
        \label{eq:p:sum_vt:xt}
    \end{equation}
    Using corollary \ref{c:max_noise} and $\alpha_t = \Theta(t^{-p})$ with $p \in (\frac{1}{2}, 1)$, the following holds in $\Eb_T'$
    \begin{equation}
        \begin{aligned}
            \sum_{t=\tau}^{T} X_{t-s+1} &\lesssim \alpha_{\tau-s} F(w_{\tau-s}) + \sum_{t=\tau-s+1}^{T-s+1} \alpha_t \InnerProd{\delta m_t}{v_t} + \left( \log \frac{2T}{\delta} \right)^{\gamma} \sum_{s \leq t=\tau}^{T} \alpha_{t-s+1}^2
            \\
            &\lesssim \left( \log \frac{3T}{\delta} \right)^2 \left( 2 \alpha_{\tau-s} + (\tau-s+1)^{-2p+1} \right)
            \\
            &\lesssim \left( \log \frac{3T}{\delta} \right)^2 (\tau-s)^{-2p+1}.
        \end{aligned}
        \label{eq:p:sum_vt:sum_xt}
    \end{equation}
    Therefore, under the event $\Eb_T'$,
    \begin{equation*}
        \sum_{s=1}^{\tau} \sum_{t=\tau}^{T} \beta^{2(s-1)} X_{t-s+1} \lesssim \left( \log \frac{3T}{\delta} \right)^2 \tau^{-2p+1} \sum_{s=1}^{\tau} \beta^{2(s-1)} \left( \frac{ \tau }{ \tau-s } \right)^{2p-1}.
        \label{eq:p:sum_vt:1}
    \end{equation*}
    Now, the ratio on the right hand side converges to 1 as $\tau \uparrow \infty$. As a consequence, the series $\sum_{s=1}^{\tau} \beta^{2(s-1)} (\tau/(\tau-s))^{2p-1}$ converges as $\tau \uparrow \infty$ by the dominated convergence theorem. Therefore, such a series in (\ref{eq:p:sum_vt:1}) is bounded by a constant independent of $\tau$.

    Now, examine the second term on the right hand side of (\ref{eq:p:sum_vt:0}). By lemma \ref{l:lsg_sb_ell}.\ref{i:l:lsg_sb_ell:nabla_ell_ell}, we have $\alpha_t^2 \, \|\nabla \ell(Z_t, w_t)\|^2 \leq \alpha_t^2 \,(a_1 \ell(Z_t, w_t) + a_2)$. By lemma \ref{l:conc_1}, we get $\sum \alpha_t^2 \,\ell(Z_t, w_t) < K_2$. Therefore, using the expression for $X_t$ in (\ref{eq:p:sum_vt:xt}), under the event $\Eb_T'$,
    \begin{equation*}
        \sum_{t=1}^{T-\tau-s+1} X_t \lesssim F(w_0) + \left( \log \frac{3T}{\delta} \right)^2 \left( \max \left( 1, \sqrt{\sum_{t=1}^{\infty} \alpha_t^4}, \sum_{t=1}^{\infty} \alpha_t^3 \right) + \left( a_1 \, K_2 + a_2 \sum_{t=0}^{\infty} \alpha_t^2 \right) \right).
    \end{equation*}
    Therefore, the claim is proved.
\end{proof}

\begin{proposition}
    With the notation in lemma \ref{l:conc_1}, if all the assumptions in proposition \ref{p:sum_vt} holds and $\gamma = 1$, then there exists a constant $K_{12}(L,\beta,\gamma,w_0,\|\alpha_t\|_{\ell^2(\N)}) > 0$ such that for any $\delta \in (0, 1)$, $T > 1$, and $\tau \in (1,T]$,
    \begin{equation*}
        \Prob \left[ \sum_{t=\tau}^{T} \InnerProd{\nabla F(w_t)}{v_t} \leq K_{12} \, \left( \log  \frac{3T}{\delta} \right)^2 \max \left( \tau^{-2p+1},\,  \beta^{2\tau} \right) \right] \geq 1 - \delta.
    \end{equation*}
    \label{p:sum_gradF_vt}
\end{proposition}
\begin{proof}
    Fix $T > 1$ and $\tau > 1$. Let $\Lambda_t = \InnerProd{ \nabla F(w_t) }{ -\alpha_t \nabla \ell(Z_t, w_t) }$. By lemma \ref{l:sum_gradF_mt} and proposition \ref{p:sum_vt}, there exists an event $\Eb \in \Fa$ such that with probability $\Prob(\Eb) \geq 1-\delta$ the following occurs:
    \begin{gather}
        \sum_{t=\tau}^{T} \|v_{t+1}\|^2 \leq K_{11} \cdot \max(\tau^{-2p+1}, \beta^{2 \tau}) \left( \log \frac{3T}{\delta} \right)^2 ~~~,~~~ \forall~ \tau \in [1,T].
        \\
        \sum_{t=\tau}^{T} \Lambda_t \leq K_{10} \cdot \alpha_{\tau} \left( \log \frac{2T}{\delta} \right)^{1+\gamma} ~~~,~~~ \forall~ \tau \in [0, T].
    \end{gather}
    Since $\gamma = 1$, i.e. $\nabla F$ is Lipschitz, $\InnerProd{\nabla F(w_t)}{v_t}$ satisfies the following inequality for any $t > 1$,
    \begin{equation*}
        \begin{aligned}
            \InnerProd{\nabla F(w_t)}{v_t} &= \InnerProd{\nabla F(w_t) - \nabla F(w_{t-1}) + \nabla F(w_{t-1})}{v_t}
            \\
            &= \InnerProd{\nabla F(w_{t-1})}{\beta v_{t-1} - \alpha_{t-1} \nabla \ell(Z_{t-1}, w_{t-1})} + \InnerProd{\nabla F(w_{t}) - \nabla F(w_{t-1})}{v_t}
            \\
            &= \beta \InnerProd{ \nabla F(w_{t-1})}{v_{t-1}}
            + \Lambda_{t-1} + \InnerProd{\nabla F(w_{t}) - \nabla F(w_{t-1})}{v_t}
            \\
            &\leq \beta \, \InnerProd{ \nabla F(w_{t-1})}{v_{t-1}} + \Lambda_{t-1} + \|v_t\|^2.
        \end{aligned}
    \end{equation*}
    Also, since $v_1 = 0$, we have $\InnerProd{\nabla F(w_1)}{v_1} = 0$. Using proposition \ref{p:ode} and switching the order of integration, under the event $\Eb$, the following occurs with probability $\Prob(\Eb) \geq 1-\delta$ for all $\tau > 1$:
    \begin{equation*}
        \begin{aligned}
            \sum_{t=\tau}^{T} \InnerProd{\nabla F(w_t)&}{v_t} \leq \sum_{s=0}^{\tau-2} \beta^s \left( \sum_{t=\tau}^T \Lambda_{t-s-1} + \|v_{t-s}\|^2 \right) + \sum_{s=\tau-2}^{T-2} \beta^s \left( \sum_{t=s+2}^T \Lambda_{t-s-1} + \|v_{t-s}\|^2 \right)
            \\
            &= \sum_{s=0}^{\tau-2} \beta^s \left( \sum_{t=\tau}^T \Lambda_{t-s-1} + \|v_{t-s}\|^2 \right) + \beta^\tau \sum_{s=0}^{T-\tau} \beta^{s-2} \left( \sum_{t=\tau}^T \Lambda_{t-\tau+1} + \|v_{t-\tau+2}\|^2 \right)
            \\
            &\lesssim \left( \log \frac{3T}{\delta} \right)^2 \left( \sum_{s=0}^{\tau-2} \beta^s \cdot \alpha_{\tau-s-1} + \beta^\tau \sum_{s=0}^{\infty} \beta^{s-2} \right)
            \\
            &\lesssim \max(\tau^{-2p+1}, \beta^\tau) \cdot \left( \log \frac{3T}{\delta} \right)^2 \left( \sum_{s=0}^{\tau-2} \beta^s \left( \frac{\tau}{\tau-s-1} \right)^p + \frac{1}{\beta^2(1-\beta)} \right).
        \end{aligned}
    \end{equation*}
    The summation on the right hand side remains bounded as $\tau \uparrow \infty$ by the dominated convergence theorem so that the summation term is bounded by a constant $c > 0$ independent of $\tau$. The coefficient $K_{12}$ must be of the order $K_{12} = O(c \cdot \max(K_{10}, K_{11}))$.
\end{proof}

\begin{theorem}
    If all assumptions in proposition \ref{p:wt_max} and proposition \ref{p:sum_gradF_vt} both holds, then there exists a constant $K_{13}(L,\beta, \gamma, w_0, w_*,\|\alpha_t\|_{\ell^2(\N)}) > 0$ such that for all $T > 1$, $T_0 \in (1,T]$, $\delta \in (0, 1)$
    \begin{equation*}
        \Prob \left( F(w_{T+1}) - F_* \leq K_{13} \left( \log \frac{3T}{\delta} \right)^2\, \max\left\{ \frac{1}{\sum_{t=T_0}^{T} \alpha_t} \,,\, T_0^{-2p+1} \, , \, \beta^{T_0} \right\} \right) \geq 1 - \delta.
    \end{equation*}
    \label{t:hp}
\end{theorem}
\begin{proof}
    Let $\Delta_t F = F(w_t) - F_*$. By proposition \ref{p:sum_weighted_risk}, proposition \ref{p:sum_vt}, lemma \ref{l:sum_gradF_mt}, and proposition \ref{p:sum_gradF_vt}, there exists an event $\Eb \in \Fa_T$ with $\Prob(\Eb) \geq 1 - \delta$ such that all of the following holds
    \begin{gather}
        \sum_{t=1}^{T} \alpha_t \,\Delta_t F \leq K_7 \left( \log \frac{2T}{\delta} \right)^{\frac{3}{2}}.
        \label{eq:t:hp:max}
        \\
        \sum_{t=\tau}^{T} \InnerProd{\nabla F(w_t)}{-\alpha_t \nabla \ell(Z_t, w_t)} \leq K_{10} \cdot \alpha_{\tau} \left( \log \frac{2T}{\delta} \right)^{1+\gamma}.
        \label{eq:t:hp:mt}
        \\
        \sum_{t=\tau}^{T} \left\{ \|v_{t+1}\|^2 + \InnerProd{\nabla F(w_t)}{v_t} \right\} \leq (K_{11} + K_{12}) \left( \log \frac{3T}{\delta} \right)^2 \max\left( \tau^{-2p+1}, \beta^{\tau} \right).
        \label{eq:t:hp:sum_vt}
    \end{gather}
    If we fix $T > 1$ and $T_0 \in (1,T]$, then under the event $\Eb$, there exists $\tau \in [T_0, T]$ such that
    \begin{equation*}
        \Delta_{\tau} F \leq \sum_{t=T_0}^{T} \frac{\alpha_t \,\Delta_t F}{\sum_{t=T_0}^{T} \alpha_t} \leq \frac{ K_7 \left( \log \frac{2T}{\delta} \right)^{\frac{3}{2}} }{ \sum_{t=T_0}^{T} \alpha_t }
    \end{equation*}
    where the middle term is the weighted average of $\Delta_t F$ from $\tilde{t}$ to $T$. Using $(\gamma, L)$-smoothness of $F$ along with estimates (\ref{eq:t:hp:max}), (\ref{eq:t:hp:mt}), and (\ref{eq:t:hp:sum_vt}), on the event $\Eb_1$, we have
    \begin{equation}
        \begin{aligned}
            \Delta_{T+1} F &\leq \Delta_{\tau} F + \sum_{t=\tau}^{T} \InnerProd{\nabla F(w_t)}{w_{t+1} - w_t} + \frac{L}{1+\gamma} \sum_{t=\tau}^{T} \|v_{t+1}\|^{1+\gamma}
            \\
            &= \Delta_{\tau} F - \sum_{t=\tau}^{T} \bigg\{ \alpha_t \, \|\nabla F(w_t)\|^2 - \psi_t \bigg\} + \beta \sum_{t=\tau}^{T}  \InnerProd{\nabla F(w_t)}{v_t} + \frac{L}{2} \sum_{t=\tau}^{T} \|v_{t+1}\|^2.
            \\
            &\lesssim K \left( \log \frac{3T}{\delta} \right)^2 \max\left\{ \frac{1}{ \sum_{t=T_0}^{T} \alpha_t }, \alpha_\tau , \max( \tau^{-2p+1}, \beta^{\tau}) \right\}.
        \end{aligned}
        \label{eq:t:hp:F_wt}
    \end{equation}
    where
    \begin{equation*}
        K = \max\left( K_7, K_{10}, \frac{L + 2\beta}{2}  (K_{11} + K_{12}) \right).
    \end{equation*}
    The proof is complete.
\end{proof}

\begin{remark}
    Theorem \ref{t:risk_hp} follows from setting $T_0 = \lfloor T/2 \rfloor$. Observe that $\sum_{t=\lfloor T/2 \rfloor}^{T} \alpha_t \geq \alpha_T \, \lfloor \frac{T}{2} \rfloor$. Therefore, setting $T_0 = \lfloor T/2 \rfloor$ in the conclusion of theorem \ref{t:hp} yields
    \begin{equation*}
        \max \left( T^{-2p+1} , T^{p-1} , \beta^{T} \right) = O(T^{\max(p-1,-2p+1)}) ~~~,~~ p \in \left(\frac{1}{2}, 1\right).
    \end{equation*}
    Therefore, theorem \ref{t:risk_hp} follows.
\end{remark}

\appendix

\section{Proof of Proposition \ref{p:weak_RS} and Lemma \ref{l:weak_RS_rate}}

\subsection{Proof of Proposition \ref{p:weak_RS}}

\begin{proof}
    Since $X_t$ is non-negative, we may drop $X_t$ from the assumed inequality on $Y_t$ and obtain $Y_t - Y_{t-1} \leq a_{t-1} Y_{t-1} - X_{t-1} + Z_{t-1} \leq a_{t-1} Y_{t-1} + Z_{t-1}$. We may sum up both sides to obtain the recursive relation
    \begin{equation}
        Y_t \leq Y_0 + \sum_{s=1}^{t} a_{s-1} Y_{s-1} + \sum_{s=1}^{t} Z_{s-1}.
        \label{eq:p:weak_RS:yt}
    \end{equation}
    Since $Z_t \in \ell^1(\N)$, the above is a Gronwall-type inequality. By proposition \ref{p:gronwall}, we get
    \begin{equation*}
        Y_t \lesssim (Y_0 + \|Z\|_{\ell^1}) \prod_{s=1}^{\infty} (1 + a_{s-1}).
    \end{equation*}
    Since $a_t \in \ell^1$, the right hand side converges. Therefore, the first and second claim are proved. To prove the second claim, repeating the same argument to produce (\ref{eq:p:weak_RS:yt}), but without dropping $X_t$, we get
    \begin{equation*}
        \sum_{s=1}^{t} X_{s-1} \leq Y_t + \sum_{s=1}^{t} X_{s-1} \leq Y_0 + \sum_{s=1}^{t} a_{s-1} Y_{s-1} + \sum_{s=1}^{t} Z_{s-1}.
    \end{equation*}
    By the first claim, $Y_t$ is uniformly bounded. Therefore, the right hand side is finite as $t \rightarrow \infty$. The proof is complete.
\end{proof}

\subsection{Proof of Lemma \ref{l:weak_RS_rate}}

\begin{proof}
    We follow \cite[lemma 19]{liuyuan2024}. Let $r \in (0, 1)$. By the mean value theorem, there exists $\delta \in (0, 1)$ such that $(t+1)^r - t^r = r (t + \delta)^{r-1} \leq r t^{r-1}$. For shorthand notation, let $r_\gamma = \frac{\gamma}{1+\gamma}$ and $a_1, a_2 \in \R_+$ be such that $a_1 t^{-p} \leq \alpha_t \leq a_2 t^{-p}$ for all $t > 0$. For a fixed $s > 0$ and any $t \geq s$, we get
    \begin{equation}
        \begin{aligned}
            \CondExp{(t+1)^r Y_{t+1}}{\Fa_s}
            &\leq
            (1 + c_1 \, \alpha_t^{1+\gamma}) \left( 1 + r \cdot t^{-1} \right) \, \CondExp{t^r Y_t}{\Fa_s}
            \\
            &+ \alpha_t^{1+\gamma} \left( t^r + r \cdot t^{r-1} \right) \left( \CondExp{Z_t}{\Fa_s} + c_2 \right)
            \\
            \\
            &\leq (1 + c_1 \, \alpha_t^{1+\gamma}) \, \CondExp{t^r Y_t}{\Fa_s} + (1 + c_1 \, \alpha_t^{1+\gamma}) \, r \, \CondExp{t^{r - 1} Y_t}{\Fa_s}
            \\
            &+ a_2 \left( t^{-p(1+\gamma)+r} +  t^{-1-p(1+\gamma)+r} \right) \left(\CondExp{Z_t}{\Fa_s} + c_2 \right).
        \end{aligned}
        \label{eq:l:weak_RS_rate:1}
    \end{equation}
    We require that $r-1 < -p$ and $-p(1+\gamma)+r < \min(-p, -1)$ so that the second and third term on the right hand side are both summable. Since $p \in (\frac{1}{1+\gamma}, 1)$, such condition holds if and only if $r < \min( 1-p, p(1+\gamma)-1 )$. Now, applying proposition \ref{p:weak_RS} yields $\limsup_{t \rightarrow \infty} \, \CondExp{t^r \cdot Y_t}{\Fa_s} < \infty$. Note that under the same condition on $r$, we may apply expectation on both sides of (\ref{eq:l:weak_RS_rate:1}) and then apply \cref{p:weak_RS} to conclude $\sup_{t \geq 0} \, \Exp[t^r \cdot Y_t] < \infty$.

    Rearranging (\ref{eq:l:weak_RS_rate:1}) and summing up from $t=s$ to $t=T-1$ yields two inequalities:
    \begin{align}
        \CondExp{T^r \cdot Y_T}{\Fa_s} - s^r \cdot Y_s &\lesssim \Exp \left[ \sum_{t=s}^{T-1} (\alpha_t^{1+\gamma} \cdot t^r + t^{r-1}) \cdot Y_t \bigg| \Fa_s \right] + \Exp \left[ \sum_{t=s}^{T-1} t^{-p(1+\gamma)+r} Z_t \bigg| \Fa_s \right].
        \label{eq:l:weak_RS_rate:2a}
        \\
        \CondExp{T^r \cdot Y_T}{\Fa_s} - s^r \cdot Y_s &\leq \Exp \left[ \sum_{t=1}^{\infty} (\alpha_t^{1+\gamma} \cdot t^r + t^{r-1}) \cdot Y_t \bigg| \Fa_s \right] + \Exp \left[ \sum_{t=1}^{\infty} t^{-p(1+\gamma)+r} Z_t \bigg| \Fa_s \right].
        \label{eq:l:weak_RS_rate:2b}
    \end{align}
    Taking absolute value on both sides of (\ref{eq:l:weak_RS_rate:2a}), the limit $T \uparrow \infty$ (with Lebesgue's monotone convergence theorem since summands are non-negative), and then $s \uparrow \infty$ yields that the right hand side is zero. Therefore, we may conclude
    \begin{equation*}
        \lim_{s \uparrow \infty} \left( \lim_{T \uparrow \infty} \bigg| \CondExp{T^r \cdot Y_T}{\Fa_s} - s^r \cdot Y_s \bigg| \right) = 0 ~~~~\text{a.s.-}\Prob.
    \end{equation*}
    This implies that $\lim_{s \uparrow \infty} s^r \, Y_s$ is $\Fa_{\infty}$-measurable. Furthermore, since $\sum_{t=1}^{\infty} \alpha_t \, Y_t < \infty$ almost surely, by \cite[theorem 3.3.1]{knopp1956}, we have $\min_{s \leq t} Y_s = o(t^{p-1})$.

    We need to conclude $\liminf_{s \uparrow \infty} s^r Y_s = 0$ a.s.-$\Prob$. This holds if we require $Y_t > 0$ for all $t \geq 0$ (counterexample can be constructed if otherwise) assuming $t^r \min_{s \leq t} Y_s \downarrow 0$ as $t \uparrow \infty$. If $Y_t > 0$ for all $t \geq 0$, then one can take a subsequence $Y_{t_n} \downarrow 0$ decreasing to zero. For instance, one can take $\{t_n : n \in \N \}$ such that $Y_{t_n} = \min_{t \leq t_n} Y_t$. Therefore, we have $t_n^r Y_{t_n} = t_n^r \min_{t \leq t_n} Y_t$ where the RHS converges to zero. This implies that there exists a subsequence of $t^r Y_t$ converging to zero, i.e. $\liminf_{s \uparrow \infty} s^r Y_s = 0$ a.s.-$\Prob$. Therefore, taking $\liminf_{s \uparrow \infty}$ on both sides of (\ref{eq:l:weak_RS_rate:2b}) yields almost surely
    \begin{equation*}
        \begin{aligned}
            \lim_{T \uparrow \infty} T^r \cdot Y_T &= \lim_{T \uparrow \infty} \liminf_{s \uparrow \infty} \bigg( \CondExp{T^r \cdot Y_T}{\Fa_s} - s^r \cdot Y_s \bigg)
            \\
            &\leq \left( \sum_{t=1}^{\infty}  (\alpha_t^{1+\gamma} \cdot t^r + t^{r-1}) \cdot Y_t \right) + \Exp \left[ \sum_{t=1}^{\infty} t^{-p(1+\gamma)+r} Z_t \bigg| \Fa_{\infty} \right]
            \\
            &< \infty.
        \end{aligned}
    \end{equation*}

    Since $\CondExp{\lim_{t \uparrow \infty} t^r Y_t}{\Fa_s}$ converges as $s \uparrow \infty$, we know
    \begin{equation*}
        \liminf_{s \uparrow \infty} \CondExp{\lim_{t \uparrow \infty} t^r Y_t}{\Fa_s} = \lim_{s \uparrow \infty} \CondExp{\lim_{t \uparrow \infty} t^r Y_t}{\Fa_s}.
    \end{equation*}
    If we take $0 < s < T$ and $s \uparrow \infty$ so that $T \uparrow \infty$ and use DCT on $\lim_{T \uparrow \infty} \CondExp{T^r Y_T}{\Fa_s}$ and the right hand side of (\ref{eq:l:weak_RS_rate:2a}), we get
    \begin{equation*}
        \lim_{T \uparrow \infty} T^r \cdot Y_T = \liminf_{s \uparrow \infty} ~\Exp \left[ \lim_{T \uparrow \infty} T^r \cdot Y_T \bigg| \Fa_s \right] \leq \lim_{s \uparrow \infty} \sum_{t=s}^{\infty}  \left( (\alpha_t^{1+\gamma} \cdot t^r + t^{r-1}) \cdot Y_t + t^{-p(1+\gamma)+r} Z_t \right) = 0.
    \end{equation*}
    This completes the proof.
\end{proof}

\section{Gronwall-type Inequalities}

\begin{proposition}
    Let $\{X_n \in \R : n \in \N_0 \}$ and $\{ Y_n \in \R : n \in \N_0 \}$ be sequences, $n_0 \in \N_0$, and $\beta > 0$. The sequence
    \begin{equation*}
        Y_{n+n_0} = \beta^{2n} y + \Indicator{\N}(n) \sum_{k=0}^{n-1} \beta^{2k} X_{n_0 + n-1-k} ~~,~~ n \geq 0
    \end{equation*}
    is the unique solution to the difference equation
    \begin{equation}
        \begin{cases}
            Y_{n_0+n} = \beta^2 \, Y_{n_0+n-1} + X_{n_0+n-1} ~,~ n > 0.
            \\
            Y_{n_0} = y.
        \end{cases}
        \label{eq:p:ode:1}
    \end{equation}
    If $Z_n$ satisfies $Z_{n_0+n} \leq \beta^2 \, Z_{n_0+n-1} + X_{n_0+n-1}$ (resp. $\geq$) and $Z_{n_0} \leq y$ (resp. $\geq$), then $Z_{n_0+n} \leq Y_{n_0 + n}$ (resp. $\geq$).
    \label{p:ode}
\end{proposition}
\begin{proof}
    We have $Y_{n_0} = \beta^{0} y = y$ and $Y_{n_0 + 1} = \beta^2 y + X_{n_0} = \beta^2 Y_{n_0} + X_{n_0}$. Moreover, for $n > 0$,
    \begin{equation*}
        \begin{aligned}
            X_{n_0+n} + \beta^2 Y_{n_0+n} &= X_{n_0+n} + \beta^{2(n+1)} y + \sum_{k=0}^{n-1} \beta^{2k + 2} X_{n_0+n-1-k}
            \\
            &= X_{n_0+n} + \beta^{2(n+1)} y + \sum_{k=1}^{n} \beta^{2k} X_{n_0+n-k}
            \\
            &= Y_{n_0+n+1}.
        \end{aligned}
    \end{equation*}
    For uniqueness, if $Z_{n_0+n}$ satisfies (\ref{eq:p:ode:1}), then $Z_{n_0+n} - Y_{n_0+n} = \beta^2 (Z_{n_0+n-1} - Y_{n_0+n-1})$ for all $n > 0$ with $Z_{n_0} - Y_{n_0} = 0$. This implies $Z_{n_0+n} - Y_{n_0+n} = (Z_{n_0} - Y_{n_0}) \,\beta^{2n} = 0$ for all $n \geq 0$.

    For the second part, we show by induction that $Z_{n_0+n} \leq Y_{n_0+n}$. Note that $Z_{n_0} \leq y = Y_{n_0}$ and $Z_{n_0 + 1} \leq \beta^2 y + X_{n_0} = Y_{n_0+1}$. For the inductive step, assume $Z_{n_0 + n} \leq Y_{n_0 + n}$ for $n \leq N$. From the first claim, $Y_{n_0+n+1} = \beta^2 Y_{n_0 + n} + X_{n_0+n}$. Therefore, we have
    \begin{equation*}
        Z_{n_0+N+1} \leq \beta^2 \, Z_{n_0+N} + X_{n_0+N} \leq \beta^2 \, Y_{n_0+N} + X_{n_0+N} = Y_{n_0+N+1}.
    \end{equation*}
    On the other hand, if $Z_{n_0 + n} \geq \beta^2 Z_{n_0 + n-1} + X_{n_0 + n}$, then we may consider
    \begin{equation*}
        \begin{cases}
            -Z_{n_0 + n} \leq \beta^2 (-Z_{n_0 + n-1}) + (-X_{n_0 + n - 1})
            \\
            -Z_{n_0} \leq -y.
        \end{cases}
    \end{equation*}
    We can apply the previous claim and get $-Z_{n_0 + n} \leq -Y_{n_0 + n}$.
\end{proof}

\begin{proposition}
    Let $A \geq 0$, $X_n \geq 0$, $c_n \geq 0$.
    $$X_n \leq A + \sum_{k=1}^{n} c_{k-1} \,X_{k-1} ~~~,~~~ X_0 \leq A~~~ \implies X_n \leq A \prod_{k=1}^{n} \left( 1 + c_{k-1} \right).$$
    \label{p:gronwall}
\end{proposition}
\begin{proof}
    Define the sequence
    \begin{equation*}
        Y_n = A + \sum_{k=1}^{n} c_{k-1} \,X_{k-1} ~~~,~~~ Y_0 = A
    \end{equation*}
    Observe that $Y_{n} - Y_{n-1} = c_{n-1} X_{n-1} \leq c_{n-1} Y_{n-1}$ for all $n \geq 1$, which implies
    \begin{equation*}
        \frac{Y_{n}}{Y_{n-1}} \leq 1 + c_{n-1} ~~,~~\forall n \geq 1.
    \end{equation*}
    The above is a telescoping product. Along with the assumption $X_n \leq Y_n$, the above argument yields the desired result.
\end{proof}
\section{Some Facts Used in Proofs}

\begin{proposition}
    If the first two properties in assumption \ref{a:loss} are both satisfied, then $F$ is convex and $(\gamma, L)$-smooth.
    \label{p:F_inherits_L}
\end{proposition}
\begin{proof}
    By the convexity of $\ell(z, \cdot)$, we get $\ell(z, a \,w_1 + (1-a) w_2) \leq a\,\ell(z, w_1) + (1-a) \ell(z, w_2)$. Applying the expectation w.r.t $\rho$ proves that $F$ is convex. Now, observe that we have
    \begin{align*}
        \|\nabla F(w_1) - \nabla F(w_2) \| &= \|\Exp_{\rho} [\nabla \ell(Z, w_1) - \nabla \ell(Z, w_2)] \|
        \\
        &\leq \Exp_{\rho} \|\nabla \ell(Z, w_1) - \nabla \ell(Z, w_2)\|
        \\
        &\leq L \,\|w_1 - w_2\|^{\gamma}.
    \end{align*}
    The proof is complete.
\end{proof}

\begin{proposition}
    If $f : \R^{d} \rightarrow \R$ is $(\gamma, L)$-smooth, then (\ref{eq:gamma_smooth}) holds.
    \label{p:gamma_smooth}
\end{proposition}
\begin{proof}
    Let $\phi(t) = f(x + t(y - x))$. Note that $\phi$ is continuously differentiable. Using the fundamental theorem of calculus on $\phi$ yields
    \begin{align*}
        f(y) &= f(x) + \int_{0}^{1} \InnerProd{\nabla f(x + t(y-x)) - \nabla f(x)}{y-x} \,dt\\
        &\leq f(x) + \int_{0}^{1} \|\nabla f(x + t(y-x)) - \nabla f(x)\| \,\|y - x\| \,dt\\
        &\leq f(x) + L \,\|y - x\|^{1+\gamma} \int_{0}^{1} t^{\gamma} \,dt.
    \end{align*}
    Integrating the right hand side yields the desired result.
\end{proof}

\begin{lemma}
    (\cite[lemma 1]{orabona2020}) Let $b_t$, $\eta_t$ be two non-negative sequences and $a_t$ be a sequence of vectors in a vector space $X$. Let $p \geq 1$ and assume $\sum_{t=1}^{\infty} \eta_t b_t^p < \infty$ and $\sum_{t=1}^{\infty} \eta_t = \infty$. Assume also that there exists $L > 0$ such that both of the following conditions hold
    \begin{enumerate}
        \item $|b_{t+\tau} - b_t| \leq L \left( \sum_{i=t}^{t+\tau} \eta_i b_i + \left\| \sum_{i=t}^{t+\tau-1} \eta_i a_i \right\| \right)$.
        \item $\left\| \sum_{i=1}^{\infty} \eta_i a_i \right\| < \infty$.
    \end{enumerate}
    Then $b_t \rightarrow 0$.
    \label{l:orabona}
\end{lemma}

\begin{lemma}
    (\cite[lemma 13]{lei2018}). If $f : \R^{d} \rightarrow \R$ is $(\gamma, L)$-smooth and convex, then
    \begin{equation*}
        \frac{ \gamma \,\|\nabla f(y) - \nabla f(x)\|^{\frac{1+\gamma}{\gamma}} }{(1 + \gamma) L^{1/\gamma}} \leq f(y) - f(x) - \InnerProd{\nabla f(x)}{y-x} \leq \frac{L \,\|y-x\|^{1+\gamma}}{1 + \gamma}.
    \end{equation*}
    \label{l:lsg_13}
\end{lemma}

\begin{lemma}
    (\cite[lemma 14]{lei2018}). If assumption \ref{a:loss} holds and $\ell(z, \cdot)$ is convex for all $z \in \Zb$, then for all $\theta \in (0, 1]$,
    \begin{equation*}
        \Exp_{\rho}[ \|\nabla \ell(Z, w) \|^{1+\theta} ] \leq 2^{\theta} L^{1/\gamma} (1 + \theta) \,[F(w) - F_*] + \frac{2^{\theta}(1 - \gamma \theta)}{1 + \gamma} + 2^{\theta}\,\Exp_{\rho}[\|\nabla \ell(Z, w_*)\|^{1 + \theta}].
    \end{equation*}
    \label{l:lsg_14}
\end{lemma}


\begin{lemma}
    Assume assumption \ref{a:loss} holds and $\ell(z, \cdot)$ is convex for all $z \in \Zb$.
    \begin{enumerate}
        \item There exists $a_0, a_1 \in \R_+$ such that $\|\nabla F(w)\|^2 \leq a_1 \,F(w) + a_0$ where $a_1$ is in (\ref{i:l:lsg_sb_ell:nabla_ell_ell}).\label{i:l:lsg_sb_ell:nabla_F}
        \item $F(w_t) \leq F_* + \frac{L}{1+\gamma} \|w_t - w_*\|^{1+\gamma}$.\label{i:l:lsg_sb_ell:F_wt}
        \item If in addition assumption \ref{a:ell_cond} holds, then there exists $a_1, a_2, a_4, a_5, a_6, a_7 \in \R_+$ such that
        \begin{enumerate}
            \item $\|\nabla\ell(z, w)\|^2 \leq a_1 \, \ell(z, w) + a_2$.\label{i:l:lsg_sb_ell:nabla_ell_ell}
            \item $\|\nabla F(w) - \nabla \ell(z, w)\|^2 \leq 6 L^2 \, \|w - w_*\|^{2\gamma} + a_4$.\label{i:l:lsg_sb_ell:dmt}
            \item $\|\nabla \ell(z, w)\|^2 \leq 2 L^2 \, \|w - w_*\|^{2\gamma} + a_5$.\label{i:l:lsg_sb_ell:nabla_ell_w}
            \item $\CondExp{\|\nabla F(w_t) - \nabla \ell(Z, w_t)\|^2}{\Fa_t} \leq a_6\,(F(w_t) - F_*) + a_7$.\label{i:l:lsg_sb_ell:dmt_F}
        \end{enumerate}
    \end{enumerate}
    \label{l:lsg_sb_ell}
\end{lemma}
\begin{proof}
    \textbf{Proof of (1).} We use the same argument as (\ref{eq:l:lsg_sb_ell:1}), but replace $\ell(z, w_*)$ with $F_*$.

    \textbf{Proof of (2).} Use $F(w_t) = F(w_t) - F_* + F_*$ and apply lemma \ref{l:lsg_13}.

    \textbf{Proof of (3.a).} Applying lemma \ref{l:lsg_13} and subsequently Young's inequality \cite[eq 4.3]{lei2018} with $p = (1+\gamma)/(1-\gamma)$ and $q = (1+\gamma)/(2\gamma)$ for $\gamma \in (0, 1]$ yields
    \begin{equation}
        \begin{aligned}
            \|\nabla \ell(z, w)\|^{2} &\leq \left( \frac{1+\gamma}{\gamma} \right)^{\frac{2\gamma}{1+\gamma}} L^{\frac{2}{1+\gamma}} \left( \ell(z, w) - \ell(z, w_*) \right)^{\frac{2\gamma}{1+\gamma}}
            \\
            &\leq \frac{1-\gamma}{1+\gamma}\left( \frac{1+\gamma}{\gamma} \cdot L^{1/\gamma} \right)^{\frac{2\gamma}{1-\gamma}} + \frac{2\gamma}{1+\gamma} \left( \ell(z, w) - \ell(z, w_*) \right)
            \\
            &\leq \frac{2\gamma}{1+\gamma} \, \ell(z, w) + \frac{1-\gamma}{1+\gamma}\left( \frac{1+\gamma}{\gamma} \cdot L^{1/\gamma} \right)^{\frac{2\gamma}{1-\gamma}} + \frac{2\gamma}{1+\gamma} \cdot \sup_{z \in \Zb}  \ell(z, w_*).
        \end{aligned}
        \label{eq:l:lsg_sb_ell:1}
    \end{equation}


    \textbf{Proof of (3.b).} Observe that if assumption \ref{a:ell_cond} and a minimizer $w_* \in \Wb$ exists and $\ell(z, \cdot)$ is convex in the second slot, then applying (\ref{i:l:lsg_sb_ell:nabla_ell_ell}) yields
    \begin{equation*}
        \begin{aligned}
            \|\nabla F(w) - \nabla \ell(z, w)\|^2 &\leq 2 \, \|\nabla F(w)\|^2 + 2 \, \|\nabla \ell(z, w)\|^2
            \\
            &\leq
            2 L^2 \, \|w_t - w_*\|^{2\gamma} + 4 \,  \|\nabla \ell(z, w_*)\|^2 + 4 \|\nabla \ell(z, w) - \nabla \ell(z, w_*)\|^2
            \\
            &\leq 6 L^2 \|w_t - w_*\|^{2\gamma} + 4 \, a_1 \, \sup_{z\in\Zb} \ell(z, w_*) + 4 \, a_2.
        \end{aligned}
    \end{equation*}

    \textbf{Proof of (3.c).} By assumption \ref{a:ell_cond} and claim (\ref{i:l:lsg_sb_ell:nabla_F}), we get
    \begin{equation*}
        \begin{aligned}
            \|\nabla \ell(z, w)\|^2 &\leq 2 \|\nabla \ell(z, w_*)\|^2 + 2 \|\nabla \ell(z, w) - \nabla\ell(z, w_*)\|^2
            \\
            &\leq \left( 2 a_1 \sup_{z\in\Zb}\ell(z, w_*) + 2 a_2 \right) + 2L^2 \|w - w_*\|^{2\gamma}
        \end{aligned}
    \end{equation*}

    \textbf{Proof of (3.d).} Another way of estimating $\|\nabla F(w) - \nabla \ell(z, w)\|^2$ is by using (\ref{i:l:lsg_sb_ell:nabla_F})
    \begin{equation*}
        \begin{aligned}
            \|\nabla F(w) - \nabla \ell(z, w)\|^2 &\leq 2 \|\nabla F(w_t)\|^2 + 2\|\nabla \ell(Z_t, w_t)\|^2
            \\
            &\leq 2 a_1 \, (F(w) - F_*) + 2 a_1 F_* + 2 a_0 + 2 a_1 \, \ell(z, w) + 2 a_2.
        \end{aligned}
    \end{equation*}
    Applying $\CondExp{\cdot}{\Fa_t}$ yields
    \begin{equation*}
        \begin{aligned}
            \CondExp{\|\delta m_t\|^2}{\Fa_t} &\leq 2 a_1 \, (F(w_t) - F_*) + 2 a_1 \, F(w_t) + 2(a_0 + a_2)
            \\
            &= 4 a_1 \, (F(w_t) - F_*) + 2 (a_0 + a_2 + a_1 \, F_*).
        \end{aligned}
    \end{equation*}
    This completes the proof.
\end{proof}

\begin{proposition}
    Let $\delta m_t$ be as in (\ref{eq:sgd2}) corresponding to the SHB iterates $w_t$ in (\ref{eq:shb}).
    \begin{enumerate}
        \item $\CondExp{\delta m_t}{\Fa_{s}} = \mathbf{0}$ and $\Exp[\delta m_s^{(i)} \,\delta m_t^{(j)}] = 0$ for all $s < t$ and $i \neq j$.
        \item $\CondExp{ \InnerProd{\nabla F(w_t)}{\delta m_t} }{ \Fa_{t} } = 0$ for all $t > 0$.
        \item If $\alpha_t$ is a sequence in $\R$, then $\Exp \left\|\sum_{s=t}^{T} \alpha_s \, \delta m_s \right\|^2 = \sum_{s=t}^{T} \alpha_s^2 \, \Exp\|\delta m_s\|^2$ for all $t \leq T$.
    \end{enumerate}
    \label{p:mt}
\end{proposition}
\begin{proof}
    By (\ref{eq:grad_F_estimator}), we have $\CondExp{\delta m_t}{\Fa_{t}} = \nabla F(w_t) - \CondExp{\nabla \ell(Z, w_t)}{\Fa_{t}} = 0$. Furthermore,
    \begin{equation*}
        \begin{aligned}
            \CondExp{\delta m_t}{\Fa_{t-1}} &=  \CondExp{\nabla F(w_t)}{\Fa_{t-1}} - \CondExp{\nabla \Exp_{\rho}[\ell(Z, w_t)]}{\Fa_{t-1}} = 0.
            \\
            \Exp[\delta m_s^{(i)} \,\delta m_t^{(j)}] &= \Exp[\delta m_s^{(i)} \,\Exp[\delta m_t^{(j)} \,|\, \Fa_{t-1} \,|\, \Fa_s]] = 0 ~~,~~ s < t.
        \end{aligned}
    \end{equation*}
    Since the standard inner product on $\R^{d}$ is a finite sum and the conditional expectation is linear, applying the first claim yields
    \begin{equation*}
        \CondExp{\InnerProd{\nabla F(w_t)}{\delta m_t}}{\Fa_{t}} = \sum_{i=1}^{d} \CondExp{\nabla F(w_t)^{(i)} \delta m_t^{(i)} }{\Fa_{t}} = \InnerProd{\nabla F(w_t)}{\CondExp{\delta m_t}{\Fa_{t}}} = 0.
    \end{equation*}
    For the last claim, we have by the first claim,
    \begin{equation*}
        \begin{aligned}
            \Exp\left\|\sum_{s=t}^{T} \alpha_s \, \delta m_t\right\|^2 &= \sum_{r,s=t}^{T} \alpha_r \cdot \alpha_s \cdot \Exp\InnerProd{ \delta m_r }{ \delta m_s }
            \\
            &= \left( \sum_{r=s=t}^{T} \alpha_s^2 \, \Exp\|\delta m_s\|^2 \right) + 2 \sum_{r < s = t}^{T} \alpha_r \cdot \alpha_s \cdot \Exp\InnerProd{ \delta m_r }{ \delta m_s }
            \\
            &= \sum_{r=s=t}^{T} \alpha_s^2 \, \Exp\|\delta m_s\|^2.
        \end{aligned}
    \end{equation*}
    This completes the proof.
\end{proof}


\bibliographystyle{plain}
\bibliography{sgd_bib}

\end{document}